%% file: RWDPPspa.tex
\documentclass[12pt]{amsart}
\usepackage{probpack}

\begin{document}
\title[simple random walk on discrete point processes]
      {Quenched invariance principle for simple random walk on discrete point processes}
\author[N. Kubota]{Naoki Kubota}
\address{Department of Mathematics, Graduate School of Science and Technology, Nihon University, Tokyo 101-8308, Japan}
\email{kubota@grad.math.cst.nihon-u.ac.jp}
\keywords{Random walk in random environment, Invariance principle, Discrete point process}
\subjclass[2010]{Primary 60F17; Secondary 60K37}
\date{}

\begin{abstract}
We consider the simple random walk on random graphs generated by discrete point processes.
This random walk moves on graphs whose vertex set is a random subset of a cubic lattice and whose edges are lines between any consecutive vertices on lines parallel to each coordinate axis.
Under the assumption that the discrete point processes are finitely dependent and stationary, we prove that the quenched invariance principle holds, i.e., for almost every configuration of the point process, the path distribution of the walk converges weakly to that of a Brownian motion. 
\end{abstract}

\maketitle

\input{intro}
\input{corrector}

\input{result}

\begin{flushleft}
 \bf Acknowledgments
\end{flushleft}
\vspace{-0.5em}
The author is grateful to
Prof. Shigenori Matsumoto, Prof. Takao Nishikawa, Prof. Ryoki Fukushima and Prof. Jun Misumi
for discussions on this problem.
I would also like to express my profound gratitude to the
reviewer for the very careful reading of the manuscript.

\bibliographystyle{jabbrv}
\bibliography{ref}
\end{document}

%% file: intro.tex
\section{Introduction}
Random walk in random environment constitutes one of the basic models of random motion in random media.
The validity of the quenched invariance principle (IP for short) for random walk in random media has been intensively investigated in recent years.
In dimensions greater than two,
Berger and Zeitouni \cite{BerZei08arXiv} proved it for ballistic random walk in random environment,
and Rassoul-Agha and Sepp\"al\"ainen \cite{MR2521407} also proved it
using different techniques from \cite{BerZei08arXiv}.
On the other hand, \cite{MR2599199,MR2354160,Mat08,MR2063376} are recent papers on the quenched IP
for random walk among random conductances.
The fourth paper treats random conductances which are bounded from both above and below.
Furthermore, the second and third papers only assume that random conductances are bounded from above,
and the first paper deals with unbounded random conductances from above.
In particular, the quenched IP for simple random walk on percolation clusters is treated
in \cite{MR2278453,MatPia07,MR2063376}.

Models in references listed above treat random walks with bounded jumps and i.i.d.\,configurations.
Our interest is now the quenched IP for models where random walks have unbounded jumps and configurations are not independent. 
We consider a simple random walk on discrete point processes on $\Z^d$, i.e., a simple random walk on graphs
whose vertex set is a random subset of $\Z^d$
and whose edges are lines between any consecutive vertices on lines parallel to each coordinate axis. 
This model was introduced in \cite{BerRos11arXiv}, and its law of large numbers and quenched central limit theorem
(CLT for short) are shown there. 
It is open that the quenched IP is valid for this model, see \cite[Section 11]{BerRos11arXiv}. 
The aim of this paper is to prove the quenched IP when the point process is finitely dependent and stationary. 

We now describe the setting in more detail. 
Let $\Omega :=\{ 0,1 \}^{\Z^d}$ and denote an element of $\Omega$ by $\omega=(\omega(x))_{x \in \Z^d}$. 
The space $\Omega$ is equipped with the canonical product $\sigma$-field $\mathcal{G}$, the canonical shift $T_y \omega (x):=\omega (x+y)$ for $x,y \in \Z^d$ and a probability measure $\Q$. 
Let $\Omega_x:=\{ \omega \in \Omega ;\omega (x)=1 \}$ for $x \in \Z^d$ and we shall assume throughout the paper that $\Q$ satisfies the following assumptions: 
\begin{itemize}
 \item[\bf (A1)] $0<\Q (\Omega_0)<1$ holds, 
 \item[\bf (A2)] There is a positive constant $\ell$ such that 
                 if $A,\,B \subset \Z^d$ satisfy $\inf \{|x-y|;x \in A,\,y \in B \} \geq \ell$ then 
                 $\sigma (\omega (x);x \in A)$ and $\sigma (\omega (x);x \in B)$ are independent, 
 \item[\bf (A3)] $\Q$ is stationary with respect to canonical shifts $(T_x)_{x \in \Z^d}$. 
\end{itemize}
Then we can define the probability measure $\P$ on $\Omega_0$ as follows: 
\begin{align*}
 \P (A):=\Q (A|\Omega_0 ),\qquad A \in \mathcal{G}.
\end{align*}
Denote the expectation with respect to $\Q$ and $\P$ by $E_{\Q}$ and $\E$, respectively. 
Let $\omega \in \Omega$ and set $\mathcal{P}(\omega):=\{ x \in \Z^d;\omega (x)=1 \}$. 
It is clear from assumptions (A1)-(A3) that $\gamma_e(\omega):=\inf \{ k \geq 1;\omega (ke)=1 \}$ has all moments under $\Q$ for all $e \in \Z^d$ with $|e|=1$. 
In particular, $\gamma_e$ is finite $\Q \hyphen \as$, thus let $N_x(\omega)$ be $2d$ nearest neighbors of $x \in \mathcal{P}(\omega)$, i.e., $\mathcal{N}_x(\omega):=\bigl\{ x+\gamma_e (T_x \omega ) e; |e|=1 \bigr\}$. 
We call a path $(x_k)_{k=0}^n$ $\mathcal{P}(\omega)$\textit{-nearest neighbor path} if $x_0 \in \mathcal{P}(\omega)$ and $x_k \in \mathcal{N}_{x_{k-1}}$ holds for every $k \in [1,n]$. 

For each $\omega \in \Omega$, the \textit{simple random walk on the discrete point process} (\textit{RWDPP} for short) is the Markov chain $((X_n)_{n=0}^\infty,(P_\omega^x)_{x \in \mathcal{P}(\omega)})$ with the state space $\mathcal{P}(\omega)$ defined by the following: for each $x \in \mathcal{P}(\omega)$, $P_\omega^x (X_0=x)=1$ and 
\begin{align}
 P_\omega^x (X_{n+1}=z|X_n=y)
 = \begin{cases}
    0 &,z \not\in \mathcal{N}_y(\omega ),\\
    \frac{1}{2d} &,z \in \mathcal{N}_y(\omega ).
   \end{cases}
 \label{eq:intro_1}
\end{align}
We call $P_\omega^x$ the \textit{quenched law} and denote the expectation with respect to $P_\omega^x$ by $E_\omega^x$. 

Our main result is that, $\P \hyphen \as$, the linear interpolation of $(X_n)_{n=0}^\infty$ 
\begin{align}
 B_n(t):=\frac{1}{\sqrt{n}}
         \left\{ X_{\lfloor tn \rfloor}+(tn-\lfloor tn \rfloor )(X_{\lfloor tn \rfloor +1}-X_{\lfloor tn \rfloor}) \right\}
         ,\qquad t \geq 0
 \label{eq:intro_10}
\end{align}
converges weakly to a Brownian motion. 
Fix $T>0$ and let $(C[0,T],\mathcal{W}_T)$ be the space of continuous functions $f:[0,T] \to \R$ equipped with the $\sigma$-field $\mathcal{W}_T$ of Borel sets relative to the topology introduced by the supremum norm. 
The precise statement of our main result is as follows: 

\begin{thm} \label{thm:intro_1}
For all $T>0$ and for $\P \hyphen a.e.\,\omega$, the law of $(B_n(t))_{0 \leq t \leq T}$ on $(C[0,T], \mathcal{W}_T)$ converges, as $n \to \infty$, weakly to the law of a Brownian motion $(B_t)_{0 \leq t \leq T}$ with a diffusion matrix $D$ that is independent of $\omega$. 
\end{thm}


\paragraph{\bf Comments on the proof}
For the proof of the theorem, we mainly follow the strategy of \cite{BerRos11arXiv}.
We write down a corrector, add it to the walk, and obtain a martingale.
Then, the Lindeberg--Feller (functional) CLT for martingales holds.
After that, all we need to do is to show that the corrector is sublinear.
In this paper, we discuss two type of sublinearities, i.e., sublinearity on average and pointwise sublinearity.
(Pointwise sublinearity is stronger than the other, see Proposition \ref{prop:corrector_2} below.)
One of the main results in \cite{BerRos11arXiv} is the quenched CLT,
and for this end it suffices to show sublinearity on average of the corrector.
However, this is not enough to establish the quenched IP.
There is an extra ingredient that is necessary.
We should either (i) show tightness of the scaled walk
or (ii) prove that the corrector has pointwise sublinearity.
The first approach is taken in the papers \cite{MR2278453,Mat08,MR2063376}
to prove the quenched IP for random walk among bounded random conductances.
For the proof of tightness of the scaled walk, heat kernel estimates are used there,
although such estimates fail to hold if
the conductance law has sufficiently heavy tails at zero,
cf. \cite{BerBisHofKoz08}.
The aforementioned fact leads to a natural question.
In the absence of heat kernel estimates, does the quenched IP still hold?
In \cite{MR2354160}, this question was affirmatively answered
by using the second approach above, i.e., establishing pointwise sublinearity of the corrector.
To prove it, they gave a sufficient condition that
is easier to verify than heat kernel estimates used in \cite{MR2278453,Mat08,MR2063376}.
In this paper, we take the second approach above, and follow the footsteps of Biskup and Prescott \cite{MR2354160}
who have used this technique to prove the quenched IP for random walk among random conductances.
In our model, heat kernel estimates have not been proved yet,
and we think that these estimates hold under the finite dependence condition.
In the future, we would like to show the quenched IP under more general conditions as in \cite{BerRos11arXiv},
i.e., stationary ergodic configurations and some moment condition for $\gamma_e(\omega)$.
However, heart kernel estimates on general ergodic configurations
are not progressed well even in random walk among random conductances, see \cite[Section 6]{MR2599199}.
Fortunately, our goal in the present work is the quenched IP,
and we take the second approach, which also prepares us for more general cases.

\paragraph{\bf Organization of the paper}
Let us now describe how the present article is organized. 
Section \ref{sect:corrector} recalls the corrector, which was introduced in \cite{MR834478} by using spectral calculus.
After that, Berger and Biskup \cite{MR2278453} adapted the construction presented there
to simple random walk on the (unique) infinite cluster of supercritical bond percolation in $\Z^d$.
Its properties were analyzed in more detail by Biskup and Prescott \cite{MR2354160}. 
Proposition \ref{prop:corrector_2} provides a sufficient condition for sublinearity of the corrector in our model. 

In Section 3 we give the proof of Theorem \ref{thm:intro_1}. 
We adopt the approach in \cite{MR2354160} and therefore it suffices to prove pointwise sublinearity of the corrector. 
Therefore, the proof of Theorem 1.1 is just a sketch, and we concentrate our effort on proving pointwise sublinearity. 
Propositions~\ref{thm:result_1}, \ref{thm:result_2}, \ref{thm:result_3} and \ref{thm:result_4} guarantee that a sufficient condition for pointwise sublinearity of the corrector is satisfied in our model. 

We close this section with some general notation. 
Let us denote $|x|_\infty:=\max_{1 \leq i \leq d} |x_i|$ for $x=(x_1,\dots,x_d) \in \Z^d$ and define $B_\infty (x,n):=\{ y \in \Z^d;|x-y|_\infty \leq n \}$. 
The $L^2(\P)$-norm for random variables is denoted by $||\cdot||_2$ and the canonical unit vectors of $\R^d$ are $e_1,\dots,e_d$. 

%% file: corrector.tex
\section{Corrector} \label{sect:corrector}
In this section, we introduce the corrector, which plays a key role in the proof of the quenched IP. 
Let us first state the following proposition. 
For the proof of this proposition we refer the reader to \cite[Sections 7 and 9]{BerRos11arXiv}. 

\begin{prop} \label{prop:corrector_1}
There exists a function $\chi:\Z^d \times \Omega_0 \to \R^d$ such that the following properties hold: 
\begin{enumerate}
 \item (Shift invariance) For $\P \hyphen a.e.\, \omega$, we have
       \begin{align*}
        \chi (x,\omega )-\chi (y,\omega )=\chi (x-y,T_y\omega ),\qquad x,\, y \in \mathcal{P}(\omega ).
       \end{align*}
 \item (Harmonicity) For $\P \hyphen a.e.\,\omega$, the function 
       \begin{align*}
        x \longmapsto \chi (x,\omega )+x
       \end{align*}
       is harmonic on $\mathcal{P}(\omega)$ with respect to the transition probability \eqref{eq:intro_1}, 
       i.e., for every $x \in \mathcal{P}(\omega)$, 
       \begin{align*}
        \frac{1}{2d} \sum_{y \in \mathcal{N}_x(\omega )} (y+\chi(y,\omega ))=x+\chi (x,\omega ).
       \end{align*}
 \item (Square integrability) There exists a positive constant $c_1$ such that 
       \begin{align*}
        \sum_{y \in \Z^d} ||(\chi (x,\cdot )-\chi (y,\cdot ))
        \1{\{ x \in \mathcal{P} \}} \1{\{ y \in \mathcal{N}_x \}}||_2^2
        \leq c_1
       \end{align*}
       holds for all $x \in \Z^d$. 
 \item (Sublinearity on average) Let $\epsilon$ be an arbitrary positive number. 
       Then, for $\P \hyphen a.e.\,\omega$, 
       \begin{align*}
        \lim_{n \to \infty} \frac{1}{n^d} \sum_{\substack{x \in \mathcal{P}(\omega)\\ |x| \leq n}}
        \1{\{ |\chi (x,\omega )| \geq \epsilon n \}} =0.
       \end{align*}
\end{enumerate}
\end{prop}

We call the above function $\chi$ the \textit{corrector}. 
It is easy to check from (ii) of Proposition \ref{prop:corrector_1} that $X_n$ is decomposed into the difference between a $P_\omega^0$-martingale $M_n$ and $\chi (X_n)$, i.e., for $n \geq 0$, 
\begin{align}
 M_n^{(\omega)}:=X_n+\chi (X_n,\omega ).
 \label{eq:corrector_15}
\end{align}
For this reason, to show the quenched IP for $(X_n)_{n=0}^\infty$ we adopt the method
which consists of using the Lindeberg--Feller (functional) CLT for martingale $(M_n)_{n=0}^\infty$ and an estimate on the corrector $\chi$.

Here, we explain the historical context of the above approach.
This approach was applied by Sidoravicius and Sznitman \cite{MR2063376} to prove the quenched IP
for simple random walk on percolation clusters in dimensions greater than three
and random walk among bounded random conductances.
They showed tightness of the scaled walk by using heat kernel estimates.
Then, the quenched IP for the random walk $(X_n)_{n=0}^\infty$ follows
from tightness
if we have that for almost every $\omega$ and  for all $t >0$,
\begin{align*}
 \lim_{n \to \infty} \frac{\chi (X_{\lfloor tn \rfloor},\omega )}{\sqrt{n}}=0\quad
 \textrm{in $P_\omega^0$-probability}.
\end{align*}
On the other hand, Berger and Biskup \cite{MR2278453} successfully implemented the proof of the quenched IP for simple random walk on percolation clusters in two dimensions.
They proved pointwise sublinearity  only in two dimensions, which is stronger than sublinearity on average
(see Proposition \ref{prop:corrector_2} below).
Using this instead of tightness plus sublinearity on average,
we can see that heat kernel estimates are not necessary for establishing the quenched IP.
However, in higher dimensions, they took the same approach as in \cite{MR2063376}.
Thus, they conjectured that the strategy taken in dimensions two is true in all dimensions.
Biskup and Prescott \cite{MR2354160} affirmatively answered this problem for random walk among bounded random conductances, which contains simple random walk on percolation clusters as a special case, for all dimensions bigger than one.
One of the main results of their paper is to give a sufficient condition for pointwise sublinearity,
which is, roughly speaking, sublinearity on average plus some estimates weaker than the heat kernel estimates.

Let us now turn to our model.
As we see above,
it is sufficient to show pointwise sublinearity of the corrector $\chi$:
\begin{align*}
 \lim_{n \to \infty} \max_{\substack{x \in \mathcal{P} (\omega )\\ |x| \leq n}}
 \frac{|\chi (x,\omega )|}{n}=0,\qquad \textnormal{for $\P \hyphen \ae\,\omega$}.
\end{align*}
The following proposition gives a sufficient condition for pointwise sublinearity of the corrector $\chi$. 
In \cite[Theorem 7.15]{Kum10lec}, a similar statement is shown for random walk among i.i.d.\,random conductances
and $C_{\infty,\alpha}$ in \cite{Kum10lec} plays the role of $\mathcal{P}(\omega)$.
Its proof is valid in the stationary and ergodic setting,
since the argument is done for each configuration satisfying suitable conditions
and we only use basic properties of martingales and Markov processes.
We thereby omit the proof and refer the reader to the proof of \cite[Theorem 7.15]{Kum10lec}.

Let $(N_t)_{t \geq 0}$ be the standard Poisson process with parameter one and we consider the continuous-time Markov chain $(Y_t:=X_{N_t})_{t \geq 0}$. 

\begin{prop} \label{prop:corrector_2}
Suppose that a function $\psi (\cdot ,\omega) :\mathcal{P}(\omega) \to \R^d$ satisfies the following conditions (i)-(v)
for $\P \hyphen a.e.\,\omega$: 
\begin{enumerate}
 \item (Harmonicity) The function 
       \begin{align*}
        x \longmapsto \psi (x,\omega )+x
       \end{align*}
       is harmonic on $\mathcal{P}(\omega)$ with respect to the transition probability \eqref{eq:intro_1}. 
 \item (Sublinearity on average) We have for every $\epsilon >0$, 
       \begin{align*}
        \lim_{n \to \infty} \frac{1}{n^d} \sum_{\substack{x \in \mathcal{P}(\omega )\\ |x| \leq n}}
        \1{\{ |\psi (x,\omega )| \geq \epsilon n \}}=0.
       \end{align*}
 \item (Polynomial growth) We have 
       \begin{align*}
        \lim_{n \to \infty} \max_{\substack{x \in \mathcal{P}(\omega )\\ |x| \leq n}}
        \frac{|\psi (x,\omega )|}{n^\theta}=0
       \end{align*}
       for some deterministic $\theta >0$. 
 \item (Diffusive upper bounds) We have for a deterministic sequence $b_n=o(n^2)$, 
       \begin{align}
        &\sup_{n \geq 1} \max_{\substack{x \in \mathcal{P}(\omega )\\ |x| \leq n}} \sup_{t \geq b_n}
         t^{d/2} P_\omega^x (Y_t=x)<\infty ,
         \label{eq:corrector_1}\\
        &\sup_{n \geq 1} \max_{\substack{x \in \mathcal{P}(\omega )\\ |x| \leq n}} \sup_{t \geq b_n}
         \frac{E_\omega^x [|Y_t-x|]}{\sqrt{t}}<\infty .
         \label{eq:corrector_2}
       \end{align}
 \item (Control of big jumps) Let $\tau_n:=\inf \{ t \geq 0;|Y_t-Y_0| \geq n \}$. 
       There exist $c_2=c_2(\omega) \geq 1$ and $N=N(\omega) \geq 1$ such that we have for all $t>0$ and $n \geq N$, 
       \begin{align*}
        \max_{\substack{x \in \mathcal{P}(\omega )\\ |x| \leq n}}
        P_\omega^x (|Y_{t \wedge {\tau_n}}-x|>c_2 n)=0.
       \end{align*}
\end{enumerate}
Under these conditions, the function $\psi$ satisfies the pointwise sublinearity
\begin{align}
 \lim_{n \to \infty} \max_{\substack{x \in \mathcal{P} (\omega )\\ |x| \leq n}}
 \frac{|\psi (x,\omega )|}{n}=0\qquad \textrm{for $\P \hyphen a.e.\,\omega$.}
 \label{eq:corrector_3}
\end{align}
\end{prop}

%% file: result.tex
\section{Proof of main result}
The aim of this section is to prove Theorem \ref{thm:intro_1}. 
Let us first give the sketch of the proof, see \cite[Sections 6.1 and 6.2]{MR2278453} or \cite[Theorem 2.1]{MR2354160} for more details. 
We now suppose the conditions 
\begin{itemize}
 \item[\bf (S)] The corrector $\chi$ satisfies the pointwise sublinearity \eqref{eq:corrector_3}, 
 \item[\bf (B)] For $\P \hyphen \ae \,\omega$ and each $n \geq 0$ there exists a positive constant $K(\omega,n)$ 
                such that $X_n \leq K(\omega ,n)$ holds $P_\omega^0 \hyphen \as$ 
\end{itemize}
Let $(M_n^{(\omega)})_{n=0}^\infty$ be a martingale in \eqref{eq:corrector_15} and set $\mathcal{F}_k:=\sigma (X_0,\dots,X_k)$. 
Fix a vector $a \in \R^d$. 
$(M_n^{(\omega)})_{n=0}^\infty$ is an $L^2(P_\omega^0)$-martingale since condition (B) implies that $X_n$ and $\chi (X_n,\omega)$ are bounded under $P_\omega^0$. 
We can define the random variable 
\begin{align*}
 V_{n,m}^{(\omega )}(\epsilon )
 := \frac{1}{n} \sum_{k=0}^m E_\omega^0
    \Bigl[ (a \cdot (M_{k+1}^{(\omega )}-M_k^{(\omega )}))^2
           \,\1{\{ |a \cdot (M_{k+1}^{(\omega )}-M_k^{(\omega )})|
                   \geq \epsilon \sqrt{n} \}} \Big| \mathcal{F}_k \Bigr]
\end{align*}
for $\epsilon >0$ and $m \leq n$. 
Denoting 
\begin{align*}
 f_K(\omega ):=E_\omega^0 \Bigl[ (a \cdot M_1^{(\omega )})^2 \,\1{\{ |a \cdot M_1^{(\omega )}| \geq K \}} \Bigr]
\end{align*}
for $K \geq 0$, then we may write from \cite[Lemma 6.1]{MR2278453}, 
\begin{align*}
 V_{n,m}^{(\omega )}(\epsilon ):=\frac{1}{n} \sum_{k=0}^m f_{\epsilon \sqrt{n}} \circ T_{X_k} (\omega ).
\end{align*}
The Markov chain on environments, $n \to T_{X_n} \omega$ is ergodic, see \cite[Theorem 3.2]{MR2278453}, thus the conditions of the Lindeberg--Feller (functional) CLT for martingale hold, see \cite[Theorem 7.7.3]{MR1609153}. 
Thereby we conclude that the random continuous function 
\begin{align*}
 t \mapsto \frac{1}{\sqrt{n}} \Bigl\{ a \cdot M_{\lfloor nt \rfloor}^{(\omega )}
           +( nt-\lfloor nt \rfloor )a \cdot (M_{\lfloor nt \rfloor +1}^{(\omega )}
                                              -M_{\lfloor nt \rfloor}^{(\omega )}) \Bigr\}
\end{align*}
converges weakly to Brownian motion with mean zero and covariance 
\begin{align*}
 \E[f_0]=\E[E_\omega^0[(a \cdot M_1^{(\omega)})^2]],
\end{align*}
which is finite from (iii) of Proposition \ref{prop:corrector_1}. 
This can be written as $a \cdot Da$ where $D$ is the matrix with coefficients 
\begin{align*}
 D_{ij}:=\E[E_\omega^0[(e_i \cdot M_1^{(\omega )})(e_j \cdot M_1^{(\omega )})]].
\end{align*}
Applying the Cram\'er-Wold device, see \cite[Theorem 2.9.5]{MR1609153}, we conclude that the linear interpolation of the map $t \mapsto M_{\lfloor nt \rfloor}^{(\omega)}/\sqrt{n}$ scales to $d$-dimensional Brownian motion with covariance matrix $D$. 
Condition (S) implies that $M_n^{(\omega)}-X_n=\chi (X_n,\omega)=o(\sqrt{n})$,
see \cite[page 110]{MR2278453} for more detailed calculations.
The same conclusion hence applies to $t \mapsto B_n(t)$ in \eqref{eq:intro_10}. 

For the completeness of the proof of Theorem \ref{thm:intro_1}, we have to show that conditions (S) and (B) are satisfied. 
Thanks to (ii) and (iv) of Proposition~\ref{prop:corrector_1}, for the proof of (S) it is enough to check conditions (iii), (iv) and (v) of Proposition~\ref{prop:corrector_2}. 
Let us first prove conditions (B) and (v) of Proposition~\ref{prop:corrector_2}. 

\begin{prop} \label{thm:result_1}
Conditions (B) and (v) of Proposition \ref{prop:corrector_2} hold. 
\end{prop}
\begin{proof}
Let us first construct a set $\hat{\Omega}_1 \subset \Omega_0$ of full $\P$-measure such that the following holds for $\omega \in \hat{\Omega}_1$: there is a positive integer $M(\omega)$ such that $\gamma_e (T_x \omega ) \leq n$ holds for all $n \geq M(\omega)$ and for all $x \in \mathcal{P}(\omega),\,e \in \Z^d$ with $|x|_\infty =n,\, |e|=1$. 
For $n \geq 1$, we define a subset $A_n$ of $\Omega_0$ by 
\begin{align*}
 A_n:=\{ \gamma_e \circ T_x>n \textrm{ for some $x \in \mathcal{P},\,e \in \Z^d$ with $|x|_\infty=n,\,|e|=1$} \} .
\end{align*}
Note that there exists a constant $c_3=c_3(d)$ such that $\#\{ x \in \Z^d;|x|_\infty =n \} \leq c_3n^{d-1}$ for $n \geq 1$.  Chebyshev's inequality then implies that for $n \geq 1$ and $\epsilon>0$, 
\begin{align*}
 \P (A_n)
 &\leq \sum_{|x|_\infty =n} \sum_{|e|=1} \P (x \in \mathcal{P},\,\gamma_e \circ T_x>n)\\
 &\leq \sum_{|x|_\infty =n} \sum_{|e|=1} \Q (\Omega_0)^{-1} \Q (\gamma_e >n,\,\omega (0)=1)\\
 &\leq c_3n^{d-1} \sum_{|e|=1} \P (\gamma_e >n)\\
 &\leq \frac{c_3n^{d-1}}{n^{d+\epsilon}} \sum_{|e|=1}\E [\gamma_e^{d+\epsilon}].
\end{align*}
It follows that the sequence $(\P (A_n))_{n=0}^\infty$ is summable, and then we can construct the desired set $\hat{\Omega}_1$ from the Borel--Cantelli lemma. 

Next we shall check condition (v) of Proposition \ref{prop:corrector_2}. 
Let $\omega \in \hat{\Omega}_1$ and define a positive integer $N(\omega)$ by 
\begin{align*}
 N(\omega ):=M(\omega )+\max \{ \gamma_e (T_x \omega);x \in B(0,M(\omega )),\,e \in \Z^d \textrm{ with } |e|=1 \} .
\end{align*}
Then the following holds for $\omega \in \hat{\Omega}_1$: 
$\mathcal{N}_y(\omega) \subset B_\infty (x,3n)$ for all $n \geq N(\omega)$ and all $x \in \mathcal{P}(\omega) \cap B_\infty (0,n),\,y \in \mathcal{P}(\omega) \cap B_\infty (x,n)$. 
This ensures that we have for $\P \hyphen \ae \,\omega$, 
\begin{align*}
 \max_{\substack{x \in \mathcal{P}(\omega )\\ |x|_\infty \leq n}} P_\omega^x (|Y_{t \wedge \tau_n}-Y_0|_\infty >3n)=0,\qquad
 t>0,\,n \geq N(\omega). 
\end{align*}
Replacing $|\cdot|_\infty$ by $|\cdot|$, we find that condition (v) of Proposition \ref{prop:corrector_2} is satisfied. 

Finally, let us prove condition (B). 
For fixed $\omega \in \hat{\Omega}_1$ and $n \geq 0$, we consider a $\mathcal{P}(\omega)$-nearest neighbor path $(0=x_0,x_1,\dots,x_n)$ starting at the origin of length $n$. 
It is clear by the choice of $N(\omega)$ that $|x_1|_\infty \leq N(\omega)$ holds. 
Thanks to the choice of $N(\omega)$, $|x_2|_\infty$ has to be less than $2N(\omega)$. 
By induction on the steps of $(0=x_0,x_1,\dots,x_n)$ we can see that for $2 \leq i \leq n$, $|x_i| \leq 2^{i-2}N(\omega)$. 
We hence have the following rough upper bound on the $|\cdot|_\infty$-norm of $x_n$: 
\begin{align*}
 |x_n|_\infty
 \leq N(\omega ) \Biggl( 1+\sum_{i=2}^n 2^{i-2} \Biggr)
 = 2^{n-1} N(\omega ),
\end{align*}
which verifies condition (B). 
\end{proof}

We next prove bound \eqref{eq:corrector_1} in condition (iv) of Proposition \ref{prop:corrector_2}. 

\begin{prop} \label{thm:result_2}
Bound \eqref{eq:corrector_1} in condition (iv) of Proposition \ref{prop:corrector_2} holds. 
\end{prop}
\begin{proof}
From \cite[Claim 5.6]{BerRos11arXiv}, there exist positive constants $c'_3$ and $c'_4$ depending only on
$d$ such that  the following holds $\P \hyphen \as$: 
there exists a positive integer $N$ such that for all $n \geq N$ and all $x,\,y \in \mathcal{P}$, 
\begin{align*}
 P_\omega^x (X_n=y) \leq \frac{c'_3}{(n-c'_4)^{d/2}}.
\end{align*}
The estimate above is the corresponding bound for the discrete-time version of $(Y_t)_{t \geq 0}$,
and hence we finish the proof by using a similar argument below the end of the proof of bound (6.10) in \cite{MR2354160}.
\end{proof}

For \eqref{eq:corrector_2} in condition (iv) of Proposition \ref{prop:corrector_2}, let us introduce some notation
and state a lemma. 
We denote the graph distance on $\mathcal{P}(\omega)$ by $d_\omega (x,y)$.
Note that the random walk $(Y_t)_{t \geq 0}$ has unbounded jumps under the Euclidean distance, however it has bounded jumps under the graph distance $d_\omega$.
In addition, our model satisfies the following regularity condition of volume growth.
Let $D(x,r):=\{ y \in \mathcal{P}(\omega);d_\omega(x,y)<r \}$.
There exists a positive constant $C$ such that
\begin{align*}
 \sum_{y \in D(x,r)}\sum_{z \in \Z^d} \1{\{ z \in \mathcal{N}_y \}}
 \leq Cr^d
\end{align*}
for all $x \in \mathcal{P}(\omega)$ and $r \geq 1$.
Indeed, the number of points in $\mathcal{P}(\omega)$ with the graph distance less than $r \geq 1$ is bounded by $2d^dr^d$.
Therefore, using the same strategy as in \cite[Proposition 3.4]{MR2094438} or \cite[Proposition 6.2]{MR2354160}, we obtain 
\begin{align}
 \sup_{z \in \Z^d} \sup_{t \geq 1} \frac{E_\omega^z[d_\omega (z,Y_t)]}{\sqrt{t}}<\infty .
 \label{eq:result_20}
\end{align}
For this reason, if $d_\omega (x,y)$ is comparable with the Euclidean distance $|x-y|$, then we can get the corresponding bound for the Euclidean distance version. 

\begin{lem} \label{lem:result_1}
There exist positive constants $\rho,\,c_5$, and $c_6$ depending only on $d$ such that 
\begin{align*}
 \Q (0,x \in \mathcal{P},\,d_\omega (0,x) \leq \rho |x|) \leq c_5 \exp \{ -c_6|x| \} .
\end{align*}
\end{lem}
\begin{proof}
We consider boxes $B_L(x):=x+\{-L,\dots,L-1 \}^d,\,L \in \N,\,x \in \Z^d$. 
We call $B_L(x)$ \textit{blocked} if the following holds: 
for all $1 \leq i \in [1,d],\,j \in [1,d] \setminus \{ i \}$ and $k_j \in [-L,L)$, there exists $m_i \in [-L,L)$ such that 
\begin{align*}
 \omega (x+k_1e_1+\dots+ m_ie_i+\dots+ k_de_d)=1.
\end{align*}
If $B_L(x)$ is not blocked, then we call it \textit{unblocked}. 
From assumptions (A1)-(A3), we can see that $p_L:=\Q (B_L(0) \textrm{ is unblocked})$ converges to zero as $L \to \infty$. 
For $L \in \N$ we set $\mathcal{C}_L(0):=\{ B_L(x);x \in 4L\Z^d \}$ and then let us introduce for $k \in [1,d]$ and distinct unit vectors $v_1,\dots,v_k \in \{ e_1,\dots,e_d \}$, 
\begin{align*}
 \mathcal{C}_L(v_1,\dots, v_k):=\{ B_L(x);x \in 2L(v_1+\dots+ v_k)+4L\Z^d \} .
\end{align*}
To simplify notation, let $\mathcal{C}_L^{(i)} \,(1 \leq i \leq 2^d)$ be an enumeration of $\mathcal{C}_L(0)$ and $\mathcal{C}_L(v_1,\dots, v_k)$'s. 
In addition, we define $\mathcal{I}_L:=\bigcup_{i=1}^{2^d} \mathcal{C}_L^{(i)}$. 
Choose $\delta :=1/2^{d+1}$ and fix integers $L \geq \ell$
(recall that $\ell$ appears in the finite dependence condition (A2)) and $n \geq 0$. 
Let $\mathcal{A}_n$ be a lattice animal on $2L \Z^d$ of size $n$, which is a connected subset of $2L \Z^d$ of $n$ vertices and containing $0$. 
The event $\Gamma (x,\mathcal{A}_n)$ is then defined as follows: 
$d_\omega (0,x)<\delta |x|_\infty/(2L)$ and a $\mathcal{P}(\omega)$-nearest neighbor path realizing $d_\omega (0,x)$, which is chosen by a deterministic algorithm, contains $n$ elements of $\mathcal{I}_L$ and is included in $\bigcup_{y \in \mathcal{A}_n} B_L(y)$. 
Furthermore, for $k \in [n/2^d,n]$ and $i \in [1,2^d]$ we define the event $G(k,i,\mathcal{A}_n)$ as follows:
$\bigcup_{y \in \mathcal{A}_n} B_L(y)$ contains exactly $k$ elements of $\mathcal{C}_L^{(i)}$ and at least $\delta n$ of its $k$ elements of $\mathcal{C}_L^{(i)}$ are unblocked. 
Note that some $\mathcal{C}_L^{(i_0)}$ has at least $n/2^d$ elements included in $\bigcup_{y \in \mathcal{A}_n} B_L(y)$,
and $\mathcal{P}(\omega)$-nearest neighbor paths realizing $d_\omega (0,x)$ intersect at least $|x|_\infty /(2L)$ elements of $\mathcal{I}_L$.
If at most $\delta n$ these elements of $\mathcal{C}_L^{(i_0)}$ are unblocked,
then the realizing path passes through at least $n/2^d-\delta n=\delta n \geq \delta |x|_\infty /(2L)$ blocked boxes in $\mathcal{C}_L^{(i_0)}$.
Therefore, one has
\begin{align*}
 \Gamma (x,\mathcal{A}_n) \subset \bigcup_{n/2^d \leq k \leq n} \bigcup_{i=1}^{2^d} G(k,i,\mathcal{A}_n).
\end{align*}
On the other hand, for each $i \in [1,2^d]$, the events $\{C \textrm{ is unblocked} \}$
for $C \in \mathcal{C}_L^{(i)} \subset \{ B_L(y);y \in \mathcal{A}_n\}$ are independent and have the same probability $p_L$. 
Set $\delta'(L):=\delta-p_L$ and 
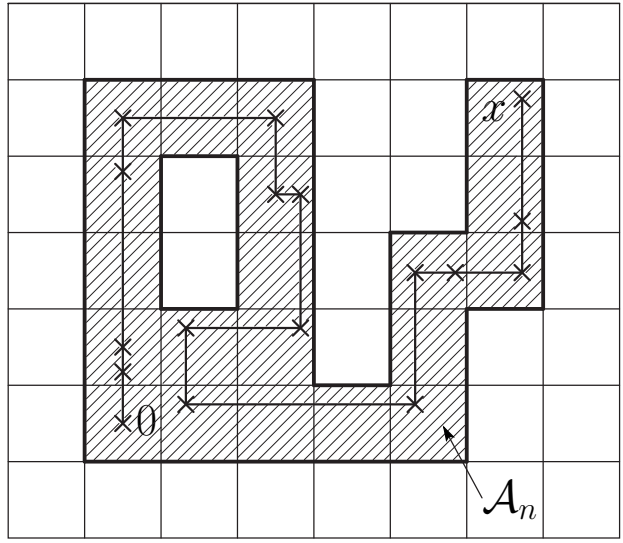
\begin{figure}[th]
\begin{center}
\input{graph_1}
\caption{A $\mathcal{P}(\omega)$-nearest neighbor path realizing $d_\omega (0,x)$ and a lattice animal $\mathcal{A}_n$. 
         The crosses stand for sites on a $\mathcal{P}(\omega)$-nearest neighbor path.}
\label{fig:result_1}
\end{center}
\end{figure}
\begin{align*}
 \alpha (L)
 := (p_L+\delta' (L)) \log \frac{p_L+\delta' (L)}{p_L}
    +(1-p_L-\delta' (L)) \log \frac{1-p_L-\delta' (L)}{1-p_L}.
\end{align*}
The Chernoff bound yields 
\begin{align*}
 \Q (G(k,i,\mathcal{A}_n)) \leq \exp \{ -\alpha (L)k \} .
\end{align*}
This ensures that 
\begin{align*}
 \Q (\{ 0,x \in \mathcal{P} \} \cap \Gamma (x,\mathcal{A}_n))
 &\leq \sum_{n/2^d \leq k \leq n} \sum_{i=1}^{2^d} \Q (G(k,i,\mathcal{A}_n))\\
 &\leq 2^dn\exp \{ -\alpha (L)\delta n \} .
\end{align*}
Since $\alpha(L) \to \infty$ as $L \to \infty$, we can find $L$ satisfying 
\begin{align*}
 (2d)^{2n}2^dn\exp \{ -\alpha (L)\delta n \} \leq 2^dne^{-2n},\qquad n \geq 1. 
\end{align*}
Recall that the number of lattice animals on $2L \Z^d$, of size $n$, containing the origin, is roughly bounded from above by $(2d)^{2n}$, see \cite[Lemma 1]{MR1241039}. 
It follows that we get 
\begin{align*}
 &\Q (0,x \in \mathcal{P},\,d_\omega (0,x)<\delta |x|_\infty /(2L))\\
 &\leq \sum_{n \geq |x|_\infty /(2L)} (2^d)^{2n} 2^dn\exp \{ -\alpha (L)\delta n \} \\
 &\leq 2^d \sum_{n \geq |x|_\infty /(2L)} \exp \{ -n \} \\
 &\leq \frac{2^d}{1-e^{-1/2}} \exp \{ -|x|_\infty /(2L) \}.
\end{align*}
Therefore, Lemma~\ref{lem:result_1} follows.
\end{proof}

After the preparation above, let us show \eqref{eq:corrector_2} in (iv) of Proposition \ref{prop:corrector_2}. 

\begin{prop} \label{thm:result_3}
Bound \eqref{eq:corrector_2} in (iv) of Proposition \ref{prop:corrector_2} holds. 
\end{prop}
\begin{proof}
We first show that there is a positive constant $c_7$ depending only on $d$ and $\hat{\Omega}_2 \subset \Omega$ of full $\Q$-measure such that the following holds for $\omega \in \hat{\Omega}_2$: 
there exists a positive integer $N(\omega)$ such that $d_\omega (z,y) \geq \rho |z-y|$ holds for all $n \geq N(\omega)$ and all $z,y \in \mathcal{P}(\omega)$ with $|z| \leq n,\,|z-y| \geq c_7 \log n$. 
Set $c_7:=2(d+2)/c_6$. 
From Lemma \ref{lem:result_1}, we have for some positive constant $c_8$ depending only on $d$ and for all sufficiently large $N$, 
\begin{align*}
 &\sum_{n=N}^\infty \sum_{\substack{z,y \in \Z^d\\ |z| \leq n \\|z-y| \geq c_7 \log n}}
  \Q (0,y-z \in \mathcal{P},\,d_\omega (0,y-z)<\rho |z-y|)\\
 &\leq \sum_{n=N}^\infty \sum_{\substack{z,y \in \Z^d\\ |z| \leq n \\|z-y| \geq c_7 \log n}} c_5 \exp \{ -c_6|z-y| \}\\
 &\leq \sum_{n=N}^\infty \sum_{|z| \leq n} \sum_{K \geq c_7 \log n} c_8 K^{d-1} \exp \{ -c_6K \} .
\end{align*}
It is easy to see that this sum converges. 
By the Borel--Cantelli lemma the assertion stated in the beginning of the proof is verified. 

Fix $\omega \in \hat{\Omega}_2$. 
We then obtain for all $n \geq N(\omega),\,|z| \leq n$ and $t \geq n$, 
\begin{align*}
 \frac{E_\omega^z[|z-Y_t|]}{\sqrt{t}}
 &\leq \frac{1}{\sqrt{t}} \bigl( \rho^{-1} E_\omega^z[d_\omega (z,Y_t)]+c_7 \log n \bigr) \\
 &\leq \rho^{-1} \frac{E_\omega^z[d_\omega (z,Y_t)]}{\sqrt{t}}+c_7 \frac{\log n}{n},
\end{align*}
which proves bound \eqref{eq:corrector_2} in (iv) of Proposition \ref{prop:corrector_2}. 
\end{proof}

Finally, we will prove that the corrector $\chi$ satisfies condition (iii) of Proposition \ref{prop:corrector_2}. 
To do this, let us introduce some notation and state a lemma. 
Denote by $E_{\theta,n}$ the event that for any $y \in \mathcal{P} \cap [-n,n]^d$, there exists a $\mathcal{P}$-nearest neighbor path $(0=z_0,z_1,\dots,z_m=y)$ from $0$ to $y$ such that $\max_{0 \leq k \leq m} |z_k|_\infty \leq n^\theta$. 

\begin{lem} \label{lem:result_x}
Suppose that the following condition (C) holds: 
\begin{itemize}
 \item[\bf (C)] For some $\theta >0$ the sequence $(\P(E_{\theta,n}^c))_{n=1}^\infty$ is summable. 
\end{itemize}
Then, condition (iii) of Proposition \ref{prop:corrector_2} holds. 
\end{lem}
\begin{proof}
For $\omega \in \Omega_0$ let 
\begin{align*}
 R_n(\omega):=\max_{\substack{x \in \mathcal{P}(\omega )\\ |x|_\infty \leq n}} |\chi (x,\omega )|.
\end{align*}
From (iii) of Proposition \ref{prop:corrector_1}, we have 
\begin{align*}
 &\E[R_n^2 \1{E_{\theta ,n}}]\\
 &\leq (4\lambda n)^2 \sum_{|x|_\infty \leq n^\theta} \sum_{y \in \Z^d}
       \E[|\chi(x+y,\cdot )-\chi (x,\cdot )|^2 \1{\{ x \in \mathcal{P} \}} \1{\{ y \in \mathcal{N}_x \}} ]\\
 &\leq c_9 n^{\theta +2}
\end{align*}
for every $n \geq 1$ and some constant $c_9=c_9(\theta)$. 
Applying Chebyshev's inequality, we obtain for $\theta'>0$, 
\begin{align*}
 \P (R_n \1{E_{\theta ,n}} \geq n^{\theta'} ) \leq c_9 n^{-2\theta' +\theta +2}.
\end{align*}
Condition (C) thereby yields that for $\theta'>(\theta +3)/2$, 
\begin{align*}
 \sum_{n=0}^\infty \P (R_n \geq n^{\theta'} )
 &\leq \sum_{n=0}^\infty \P (R_n \1{E_{\theta ,n}} \geq n^{\theta'} )+\sum_{n=0}^\infty \P (E_{\theta ,n}^c)\\
 &\leq 1+c_9 \sum_{n=1}^\infty n^{-2\theta' +\theta +2}+\sum_{n=0}^\infty \P (E_{\theta ,n}^c)<\infty ,
\end{align*}
which proves condition (iii) of Proposition \ref{prop:corrector_2} by the Borel--Cantelli lemma. 
\end{proof}

Due to Lemma \ref{lem:result_x}, for the proof of condition (iii) of Proposition \ref{prop:corrector_2} it suffices to check condition (C).
For $x=(x_1,\dots,x_d) \in \Z^d$ and $2 \leq i \leq d$, let us denote by $F_{i,n}(x)$ the event that for any $y \in \mathcal{P} \cap [-n,n]^i \times \{ (x_{i+1},\dots,x_d) \}$, there exists a $\mathcal{P}$-nearest neighbor self-avoiding path $(x=r_0,r_1,\dots,r_m=y)$ from $x$ to $y$ such that $r_k \in [-n,n]^i \times \{ (x_{i+1},\dots,x_d) \},\,1 \leq k \leq m$. 
In addition,
let $G_{i,n}(x)$ be the version of $F_{i,n}(x)$
with $[-n,n]^i \times \{ (x_{i+1},\dots,x_d) \}$ replaced by
$\{ x_1 \} \times [-n,n]^i \times \{ (x_{i+2},\dots,x_d) \}$, i.e.,
$G_{i,n}(x)$ is the event that
for any $y \in \mathcal{P} \cap \{ x_1 \} \times [-n,n]^i \times \{ (x_{i+2},\dots,x_d) \}$,
there exists a $\mathcal{P}$-nearest neighbor self-avoiding path
$(x=r_0,r_1,\dots,r_m=y)$ from $x$ to $y$
such that $r_k \in \{ x_1 \} \times [-n,n]^i \times \{ (x_{i+2},\dots,x_d) \},\,1 \leq k \leq m$.

\begin{prop} \label{thm:result_4}
Condition (C) is satisfied, and therefore condition (iii) of Proposition \ref{prop:corrector_2} holds. 
\end{prop}
\begin{proof}
Let $\rho:=1-\Q(\Omega_0)$. 
Suppose that for some $i \in [2,d-1]$ there exist positive constants $C_1^{(i)},C_2^{(i)}$ such that for all $n$ large enough and all $x=x_1e_1+\dots+x_ie_i \in \Z^d \cap [-n,n]^d$, 
\begin{align}
 \Q (F_{i,n}(x)^c \cap \Omega_x )+\Q (G_{i,n}(x)^c \cap \Omega_x ) \leq C_1^{(i)} \rho^{C_2^{(i)}n}.
 \label{eq:result_22}
\end{align}
Denote by $H_{i,n}(z)$ the event that for all $-n \leq k \leq n$ there is a site $w=ke_1+w_2e_2+\dots+w_ie_i+z_{i+1}e_{i+1}$ such that $w \in \mathcal{P}$ and $-n \leq w_j \leq n$ for $2 \leq j \leq i$. 
A straightforward calculation then shows that we get, for $z=z_1e_1+\dots+z_{i+1}e_{i+1} \in \Z^d \cap [-n,n]^d$, 
\begin{align}
\begin{split}
 \Q (F_{i+1,n}(z)^c \cap \Omega_z )
 &\leq \Q (F_{i,n}(z)^c \cap \Omega_z )+\Q (H_{i,n}(z)^c)\\
 &\quad
       +\Q (F_{i+1,n}(z)^c \cap \Omega_z \cap F_{i,n}(z) \cap H_{i,n}(z)).
\end{split}
\label{eq:result_21}
\end{align}
By \eqref{eq:result_22} the first term of the right-hand side of \eqref{eq:result_21} is less than $C_1^{(i)} \rho^{C_2^{(i)}n}$. 
Moreover, it is easy to see that there are positive constants $c_{10}$ and $c_{11}$ such that for all sufficiently large $n$, the second term of the right-hand side of \eqref{eq:result_21} is bounded from above by $c_{10}\rho^{c_{11}n}$. 
Let us estimate the third term of the right-hand side of \eqref{eq:result_21}. 
On the event $H_{i,n}(z)$, there are sites $u_k \in \mathcal{P},\,-n \leq k \leq n$ such that 
\begin{align*}
 u_k \in \mathcal{U}(i,n,k):=\{ ke_1+v_2e_2+\dots+v_ie_i+z_{i+1}e_{i+1};-n \leq v_2,\dots,v_i \leq n \} .
\end{align*}
Note that, on the event 
\begin{align*}
 I_{i,n}(z):=\Omega_z \cap F_{i,n}(z) \cap H_{i,n}(z) \cap \bigcap_{k=-n}^n (G_{i,n}(u_k) \cap \Omega_{u_k}),
\end{align*}
for each $y=y_1e_1+\dots+y_{i+1}e_{i+1} \in [-n,n]^d \cap \mathcal{P}$ the site $u_{y_1}$ is connected to $y$ by a $\mathcal{P}$-nearest neighbor path included in $\{ y_1 \} \times [-n,n]^i \times \{ 0 \}^{d-i-1}$ and is connected to $z$ by a $\mathcal{P}$-nearest neighbor path included in $[-n,n]^i \times \{ 0 \}^{d-i}$. 
For this reason, the event $F_{i+1,n}(z)$ occurs on the event $I_{i,n}(z)$. 
We hence obtain 
\begin{align*}
 &\Q (F_{i+1,n}(z)^c \cap \Omega_z \cap F_{i,n}(z) \cap H_{i,n}(z))\\
 &\leq \sum_{\substack{(u_{-n},\dots, u_n)\\ \in \mathcal{U}(i,n,-n) \times\dots\times \mathcal{U}(i,n,n)}}
       \Q \Biggl( F_{i+1,n}(z)^c \cap \Omega_z \cap F_{i,n}(z) \cap H_{i,n}(z) \cap \bigcap_{k=-n}^n \Omega_{u_k} \Biggr) \\
 &\leq \sum_{\substack{(u_{-n},\dots, u_n)\\ \in \mathcal{U}(i,n,-n) \times\dots\times \mathcal{U}(i,n,n)}}
       \Biggl( \Q (F_{i+1,n}(z)^c \cap I_{i,n}(z))
               +\sum_{j=-n}^n \Q (G_{i,n}(u_j)^c \cap \Omega_{u_j}) \Biggr) \\
 &\leq C_1^{(i)}(2n+1)^{i+1} \rho^{C_2^{(i)}n},
\end{align*}
and therefore 
\begin{align*}
 \Q (F_{i+1,n}(z)^c \cap \Omega_z )
 &\leq C_1^{(i)} \rho^{C_2^{(i)}n}+c_{10}\rho^{c_{11}n}+C_1^{(i)}(2n+1)^{i+1} \rho^{C_2^{(i)}n}
\end{align*}
is shown. 
By the same argument as above, we can estimate $\Q (G_{i+1,n}(z)^c \cap \Omega_z )$, so there are some positive constants $C_1^{(i+1)}$ and $C_2^{(i+1)}$ such that for all $n$ large enough and all $z=z_1e_1+\dots+z_{i+1}e_{i+1} \in \Z^d \cap [-n,n]^d$, 
\begin{align*}
 \Q (F_{i+1,n}(z)^c \cap \Omega_z )+\Q (G_{i+1,n}(z)^c \cap \Omega_z )
 \leq C_1^{(i+1)} \rho^{C_2^{(i+1)}n}.
\end{align*}

By induction on $i$, it is enough to show \eqref{eq:result_22} in the case $i=2$. 
To simplify notation, let $d=2$. 
We treat only the estimate for $\Q (F_{2,n}(x)^c \cap \Omega_x ),\,x=x_1e_1+x_2e_2 \in \Z^2 \cap [-n,n]^2$, since an analogous statement with $\Q (G_{2,n}(x)^c \cap \Omega_x )$ instead of $\Q (F_{2,n}(x)^c \cap \Omega_x )$ can be derived similarly. 
Without loss of generality we can assume $x_2 \geq 0$. 
To this end, for $n \geq 1,\,u \in [0,n] \cap \Z,\,0<\delta<1$ and $L \geq 1$ let us introduce the events 
\begin{align*}
 &\Lambda_0 (u,L):=\{ (u-j)e_2 \in \mathcal{P} \textrm{ for some $1 \leq j \leq L$} \} ,\\
 &\Lambda_1 (u,\delta ,L,n)
  := \biggl\{ \frac{1}{2n+1} \sum_{i=-n}^n \1{\Lambda_0 (u,L)} \circ T_{ie_1}>1-\delta \biggr\} ,
\end{align*}
and $\Lambda_2 (u,L,n)$ defined as follows: 
for all $\ell \leq j \leq L$, there is an $0 \leq i_j \leq n$ such that 
\begin{align*}
 \omega (i_je_1)=\omega (i_je_1+(u-j)e_2)=1, 
\end{align*}
and for all $0 \leq m \leq \ell -1$, there is $k_m \in [0,n]$ such that 
\begin{align*}
 \omega (k_me_1+(u-m)e_2)=\omega (k_me_1+(u-(m+\ell ))e_2)= 1.
\end{align*}
The event $\Lambda_1 (u,\delta ,L,n)$ means that almost all vertical lines are blocked in the slab $[-n,n] \times [u-L,u]$ and the event $\Lambda_2 (u,L,n)$ means that all lines in the slab $[-n,n] \times [u-L,u]$ are connected by $\mathcal{P}$-nearest neighbor paths included in $[-n,n]^2$. 
We will estimate $\Q(\Lambda_1(u,\delta ,L,n)^c)$. 
Let $0<\delta<\Q(\Omega_0)$ and we choose $0<\delta_1<\delta$. 
Note that if 
\begin{align}
 n>\half \biggl( \frac{(1-\delta_1 )\ell}{\delta-\delta_1}-1 \biggr) ,
 \label{eq:invariance_16}
\end{align}
then 
\begin{align*}
 \frac{2n+1}{\lfloor (2n+1)/\ell \rfloor} \,\frac{1-\delta}{\ell}<1-\delta_1 .
\end{align*}
Let $I_k:=\{ -n+k+m \ell ;0\leq m \leq \lfloor (2n+1)/\ell \rfloor -1 \}$ for $0 \leq k \leq \ell -1$. 
We get for any $n$ satisfying \eqref{eq:invariance_16}, 
\begin{align*}
 &\Q (\Lambda_1 (u,\delta ,L,n)^c)\\
 &\leq \Q \Biggl( \frac{1}{2n+1} \sum_{k=0}^{\ell -1} \sum_{i \in I_k} \1{\Lambda_0 (u,L)} \circ T_{ie_1}
                  \leq 1-\delta \Biggr) \\
 &\leq \sum_{k=0}^{\ell -1} \Q \Biggl( \frac{1}{2n+1}
       \sum_{i \in I_k} \1{\Lambda_0 (u,L)} \circ T_{ie_1} \leq \frac{1-\delta}{\ell} \Biggr) \\
 &\leq \ell \,\Q \Biggl( \frac{1}{\lfloor (2n+1)/\ell \rfloor} \sum_{i \in I_0} \1{\Lambda_0 (x,L)} \circ T_{ie_1}
                         \leq \frac{2n+1}{\lfloor (2n+1)/\ell \rfloor} \,\frac{1-\delta}{\ell} \Biggr) \\
 &\leq \ell \,\Q \Biggl( \frac{1}{\lfloor (2n+1)/\ell \rfloor} \sum_{i \in I_0}
                         \1{\Lambda_0 (u,L)} \circ T_{ie_1}<1-\delta_1 \Biggr) .
\end{align*}
It is clear from (A2) and (A3) that we can take numbers $0<\delta_2<1$ and $L \geq 2\ell-1$ with $1-\delta_1 \leq (1-\delta_2) \,\Q (\Lambda_0(0,L))$. 
Thus, $\Q (\Lambda_1 (u,\delta ,L,n)^c)$ is smaller than 
\begin{align}
 \ell \,\Q \Biggl( \frac{1}{\lfloor (2n+1)/\ell \rfloor} \sum_{i \in I_0} \1{\Lambda_0 (u,L)} \circ T_{ie_1}
                   <(1-\delta_2) \,\Q (\Lambda_0 (0,L)) \Biggr) .
 \label{eq:invariance_17}
\end{align}
Using the Chernoff bound, we estimate \eqref{eq:invariance_17} from above by 
\begin{align*}
 \ell \exp \biggl\{ -\frac{\delta_2^2 \,\Q (\Lambda_0 (0,L))}{2} \biggl\lfloor \frac{2n+1}{\ell} \biggr\rfloor \biggr\} .
\end{align*}
It follows that 
\begin{align}
 \Q (\Lambda_1 (u,\delta ,L,n)^c)
 \leq \ell \exp \biggl\{ -\frac{\delta_2^2 \,\Q (\Lambda_0 (0,L))}{2} \biggl\lfloor \frac{2n+1}{\ell} \biggr\rfloor \biggr\} .
 \label{eq:invariance_18}
\end{align}
holds for all $n$ satisfying \eqref{eq:invariance_16}. 

Let us next estimate $\Q (\Lambda_2 (x,L,n)^c)$. 
Noting that by assumption (A2) $\omega (k \ell e_1),\,k \geq 0$ are independent under $\Q$, we have for $n \geq \ell$, 
\begin{align}
\begin{split}
 &\Q (\Lambda_2 (u,L,n)^c)\\
 &\leq \sum_{j=\ell}^L \Q \Bigl( \textrm{for every $0 \leq i \leq n/\ell$,}\\
 &\qquad \qquad \ \ 
                                 \textrm{$\omega (i\ell e_1)=\omega (i\ell e_1+(u-j)e_2)=1$ fails} \Bigr)\\
 &\quad
       +\sum_{m=0}^{\ell -1}
        \Q \Bigl( \textrm{for every $0 \leq k \leq n/\ell$, $\omega (k\ell e_1+(u-m)e_2)$}\\
 &\qquad \qquad \quad \ \ 
            \textrm{$=\omega (k\ell e_1+(u-(m+\ell ))e_2)= 1$ fails} \Bigr)\\
 &\leq (L-\ell +1)(1-\Q (\Omega_0 )^2)^{n/\ell}+\ell (1-\Q (\Omega_0 )^2)^{n/\ell}\\
 &= (L+1)(1-\Q (\Omega_0 )^2)^{n/\ell}.
\end{split}
\label{eq:invariance_19}
\end{align}

\begin{figure}[t]
\begin{center}
\input{graph_2}
\caption{The event $E_{1,n}(x)^c \cap \Omega_x \cap \Lambda_1 (x_2,\delta ,L,n) \cap \Lambda_2 (x_2,L,n)$. 
         The crosses standing for $L$ horizontal lines are connected to $x$ by $P(\omega)$-nearest neighbor paths 
         included in $[-n,n]^2$. 
         There exists some site $y \in \mathcal{P}(\omega) \cap [-n,n]^2$ such that 
         the site $y$ is not connected to $x$ and solid vertical lines cannot have points in the slab $[-n,n] \times [x_2-L,x_2]$.
         }
\label{fig:result_2}
\end{center}
\end{figure}
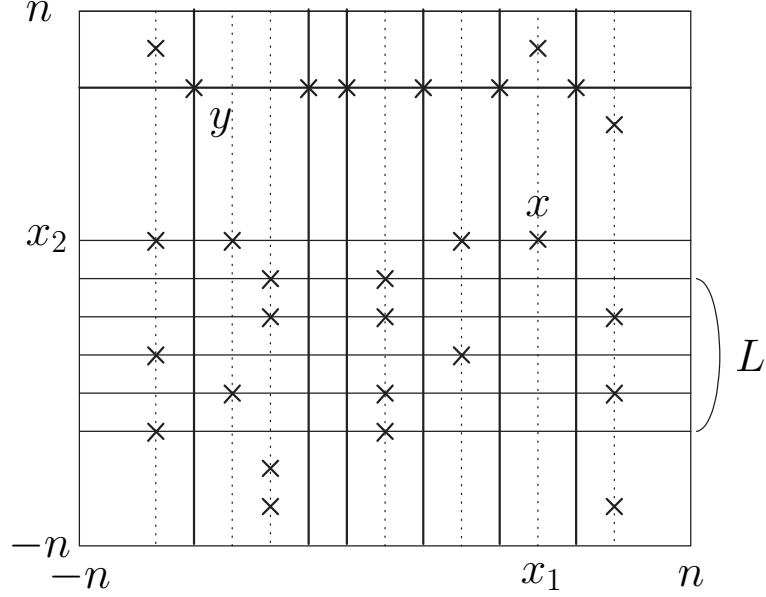
We shall estimate the left-hand side of \eqref{eq:result_22}. 
Bounds \eqref{eq:invariance_18} and \eqref{eq:invariance_19} ensure that 
\begin{align}
\begin{split}
 &\Q (E_{1,n}(x)^c \cap \Omega_x )\\
 &\leq \Q (\Lambda_1 (x_2,\delta ,L,n)^c)+\Q (\Lambda_2 (x_2,L,n)^c)\\
 &\quad
       +\Q (E_{1,n}(x)^c \cap \Omega_x \cap \Lambda_1 (x_2,\delta ,L,n) \cap \Lambda_2 (x_2,L,n))\\
 &\leq \exp \biggl\{ -\frac{\delta_2^2 \,\Q (\Lambda_0 (L))}{2} \biggl\lfloor \frac{n+1}{\ell} \biggr\rfloor \biggr\}
       +L(1-\Q (\Omega_0 )^2)^{n/\ell}\\
 &\quad
       +\Q (E_{1,n}(x)^c \cap \Omega_x \cap \Lambda_1 (x_2,\delta ,L,n) \cap \Lambda_2 (x_2,L,n)).
\end{split}
\label{eq:invariance_20}
\end{align}
holds for $x=x_1e_1+x_2e_2 \in \Z^2 \cap [-n,n]^2$. 
We now prove that there is $y_1e_1+y_2e_2 \in \mathcal{P} \cap [-n,n]^2$ such that 
\begin{align}
 \sum_{i=-n}^n \1{\{ \omega (ie_1+y_2e_2)=1 \}}<\delta (2n+1)
 \label{eq:resupt_10}
\end{align}
holds on the event $E_{1,n}(x)^c \cap \Omega_x \cap \Lambda_1 (x_2,\delta ,L,n) \cap \Lambda_2 (x_2,L,n)$. 
We suppose that on the event $\Omega_x \cap \Lambda_1 (x_2,\delta ,L,n) \cap \Lambda_2 (x_2,L,n)$, bound \eqref{eq:resupt_10} fails for all $y_1e_1+y_2e_2 \in \mathcal{P} \cap [-n,n]^2$. 
On the event $\Omega_x \cap \Lambda_1 (x_2,\delta ,L,n)$, there exist at least $\lceil (1-\delta)(2n+1) \rceil$ vertical lines contained in $[-n,n]^2$ such that each of these vertical lines has a site in $\mathcal{P}$ standing for the slab $[-n,n] \times [x_2-L,x_2]$. 
Moreover, on the event $\Omega_x \cap \Lambda_2 (x_2,L,n)$, all lines in the slab $[-n,n] \times [x_2-L,x_2]$ are connected to $x$ by $\mathcal{P}$-nearest neighbor paths included in $[-n,n]^2$. 
It follows that on the event $\Omega_x \cap \Lambda_1 (x_2,\delta ,L,n) \cap \Lambda_2 (x_2,L,n)$, at least $\lceil (1-\delta)(2n+1) \rceil$ vertical lines contained in $[-n,n]^2$ are connected to $x$ by $\mathcal{P}$-nearest neighbor paths. 
On the other hand, on the event $\Omega_x \cap \Lambda_1 (x_2,\delta ,L,n) \cap \Lambda_2 (x_2,L,n)$ there is a horizontal line contained in $[-n,n]^2$ such that it does not have sites connected to $x$ by a $\mathcal{P}$-nearest neighbor path included in $[-n,n]^2$ and this horizontal line has at least $\lceil \delta (2n+1) \rceil$ sites in $\mathcal{P}$. 
This means that the number of vertical lines contained in the box $[-n,n]^2$ has to be strictly greater than $(1-\delta)(2n+1)+\delta (2n+1)=2n+1$, which contradicts that the number of vertical lines contained in the box $[-n,n]^2$ is equal to $2n+1$. 
Therefore, on the event $\Omega_x \cap \Lambda_1 (x_2,\delta ,L,n) \cap \Lambda_2 (x_2,L,n)$, bound \eqref{eq:resupt_10} holds for some $y_1e_1+y_2e_2 \in \mathcal{P} \cap [-n,n]^2$. 

The same argument as in the proof of \eqref{eq:invariance_18} implies 
\begin{align}
\begin{split}
 &\Q (E_{1,n}(x)^c \cap \Omega_x \cap \Lambda_1 (x_2,\delta ,L,n) \cap \Lambda_2 (x_2,L,n))\\
 &\leq (2n+1)^2 \Q \biggl( \sum_{i=-n}^n \1{\{ \omega (ie_1)=1 \}}<\delta (2n+1) \biggr)\\
 &\leq \ell (2n+1)^2
       \exp \biggl\{ -\frac{\delta_3^2 \,\Q (\Omega_0 )}{2} \biggl\lfloor \frac{2n+1}{\ell} \biggr\rfloor \biggr\}
\end{split}
\label{eq:invariance_21}
\end{align}
for all sufficient large $n$ and some $0<\delta_3<1$. 
According to \eqref{eq:invariance_20} and \eqref{eq:invariance_21}, bound \eqref{eq:result_22} holds for $i=2$. 
\end{proof}

%% file: graph_1.tex
\unitlength 0.1in
\begin{picture}( 32.0000, 28.0000)(  6.0000,-32.0000)
\put(32.1000,-9.1000){\makebox(0,0)[rt]{\mbox{\Large $x$}}}%
\put(12.7000,-26.6000){\makebox(0,0)[lb]{\mbox{\Large $0$}}}%
\put(30.8000,-31.0000){\makebox(0,0)[lb]{\mbox{\Large $\mathcal{A}_n$}}}%
%
{\color[named]{Black}{%
\special{pn 8}%
\special{pa 600 400}%
\special{pa 3800 400}%
\special{pa 3800 3200}%
\special{pa 600 3200}%
\special{pa 600 400}%
\special{pa 3800 400}%
\special{fp}%
}}%
%
{\color[named]{Black}{%
\special{pn 8}%
\special{pa 1000 3200}%
\special{pa 1000 400}%
\special{fp}%
\special{pa 1400 400}%
\special{pa 1400 3200}%
\special{fp}%
\special{pa 1800 3200}%
\special{pa 1800 400}%
\special{fp}%
\special{pa 2200 400}%
\special{pa 2200 3200}%
\special{fp}%
\special{pa 2600 3200}%
\special{pa 2600 400}%
\special{fp}%
\special{pa 3000 400}%
\special{pa 3000 3200}%
\special{fp}%
\special{pa 3400 3200}%
\special{pa 3400 400}%
\special{fp}%
\special{pa 3800 800}%
\special{pa 600 800}%
\special{fp}%
\special{pa 600 1200}%
\special{pa 3800 1200}%
\special{fp}%
\special{pa 3800 1600}%
\special{pa 600 1600}%
\special{fp}%
\special{pa 600 2000}%
\special{pa 3800 2000}%
\special{fp}%
\special{pa 3800 2400}%
\special{pa 600 2400}%
\special{fp}%
\special{pa 600 2800}%
\special{pa 3800 2800}%
\special{fp}%
}}%
%
{\color[named]{Black}{%
\special{pn 13}%
\special{pa 1160 2560}%
\special{pa 1240 2640}%
\special{fp}%
\special{pa 1240 2560}%
\special{pa 1160 2640}%
\special{fp}%
}}%
%
{\color[named]{Black}{%
\special{pn 13}%
\special{pa 1160 2290}%
\special{pa 1240 2370}%
\special{fp}%
\special{pa 1240 2290}%
\special{pa 1160 2370}%
\special{fp}%
}}%
%
{\color[named]{Black}{%
\special{pn 13}%
\special{pa 1160 2160}%
\special{pa 1240 2240}%
\special{fp}%
\special{pa 1240 2160}%
\special{pa 1160 2240}%
\special{fp}%
}}%
%
{\color[named]{Black}{%
\special{pn 13}%
\special{pa 1160 1240}%
\special{pa 1240 1320}%
\special{fp}%
\special{pa 1240 1240}%
\special{pa 1160 1320}%
\special{fp}%
}}%
%
{\color[named]{Black}{%
\special{pn 13}%
\special{pa 1160 960}%
\special{pa 1240 1040}%
\special{fp}%
\special{pa 1240 960}%
\special{pa 1160 1040}%
\special{fp}%
}}%
%
{\color[named]{Black}{%
\special{pn 13}%
\special{pa 1960 960}%
\special{pa 2040 1040}%
\special{fp}%
\special{pa 2040 960}%
\special{pa 1960 1040}%
\special{fp}%
}}%
%
{\color[named]{Black}{%
\special{pn 13}%
\special{pa 1960 1360}%
\special{pa 2040 1440}%
\special{fp}%
\special{pa 2040 1360}%
\special{pa 1960 1440}%
\special{fp}%
}}%
%
{\color[named]{Black}{%
\special{pn 13}%
\special{pa 2090 1360}%
\special{pa 2170 1440}%
\special{fp}%
\special{pa 2170 1360}%
\special{pa 2090 1440}%
\special{fp}%
}}%
%
{\color[named]{Black}{%
\special{pn 13}%
\special{pa 2090 2060}%
\special{pa 2170 2140}%
\special{fp}%
\special{pa 2170 2060}%
\special{pa 2090 2140}%
\special{fp}%
}}%
%
{\color[named]{Black}{%
\special{pn 13}%
\special{pa 1490 2060}%
\special{pa 1570 2140}%
\special{fp}%
\special{pa 1570 2060}%
\special{pa 1490 2140}%
\special{fp}%
}}%
%
{\color[named]{Black}{%
\special{pn 13}%
\special{pa 1490 2460}%
\special{pa 1570 2540}%
\special{fp}%
\special{pa 1570 2460}%
\special{pa 1490 2540}%
\special{fp}%
}}%
%
{\color[named]{Black}{%
\special{pn 13}%
\special{pa 2690 2460}%
\special{pa 2770 2540}%
\special{fp}%
\special{pa 2770 2460}%
\special{pa 2690 2540}%
\special{fp}%
}}%
%
{\color[named]{Black}{%
\special{pn 13}%
\special{pa 2690 1770}%
\special{pa 2770 1850}%
\special{fp}%
\special{pa 2770 1770}%
\special{pa 2690 1850}%
\special{fp}%
}}%
%
{\color[named]{Black}{%
\special{pn 13}%
\special{pa 2900 1770}%
\special{pa 2980 1850}%
\special{fp}%
\special{pa 2980 1770}%
\special{pa 2900 1850}%
\special{fp}%
}}%
%
{\color[named]{Black}{%
\special{pn 13}%
\special{pa 3250 1500}%
\special{pa 3330 1580}%
\special{fp}%
\special{pa 3330 1500}%
\special{pa 3250 1580}%
\special{fp}%
}}%
%
{\color[named]{Black}{%
\special{pn 13}%
\special{pa 3250 1770}%
\special{pa 3330 1850}%
\special{fp}%
\special{pa 3330 1770}%
\special{pa 3250 1850}%
\special{fp}%
}}%
%
{\color[named]{Black}{%
\special{pn 13}%
\special{pa 3250 860}%
\special{pa 3330 940}%
\special{fp}%
\special{pa 3330 860}%
\special{pa 3250 940}%
\special{fp}%
}}%
%
{\color[named]{Black}{%
\special{pn 13}%
\special{pa 1200 2600}%
\special{pa 1200 1000}%
\special{pa 2000 1000}%
\special{pa 2000 1400}%
\special{pa 2130 1400}%
\special{pa 2130 2100}%
\special{pa 1530 2100}%
\special{pa 1530 2500}%
\special{pa 2730 2500}%
\special{pa 2730 1810}%
\special{pa 3290 1810}%
\special{pa 3290 910}%
\special{pa 3290 910}%
\special{pa 3290 910}%
\special{pa 3290 910}%
\special{fp}%
}}%
%
{\color[named]{Black}{%
\special{pn 20}%
\special{pa 1000 2800}%
\special{pa 1000 800}%
\special{fp}%
\special{pa 1000 800}%
\special{pa 2200 800}%
\special{fp}%
\special{pa 2200 800}%
\special{pa 2200 2400}%
\special{fp}%
\special{pa 2200 2400}%
\special{pa 2600 2400}%
\special{fp}%
\special{pa 2600 2400}%
\special{pa 2600 1600}%
\special{fp}%
\special{pa 2600 1600}%
\special{pa 3000 1600}%
\special{fp}%
\special{pa 3000 1600}%
\special{pa 3000 800}%
\special{fp}%
\special{pa 3000 800}%
\special{pa 3400 800}%
\special{fp}%
\special{pa 3400 800}%
\special{pa 3400 2000}%
\special{fp}%
\special{pa 3400 2000}%
\special{pa 3000 2000}%
\special{fp}%
\special{pa 3000 2000}%
\special{pa 3000 2800}%
\special{fp}%
\special{pa 3000 2800}%
\special{pa 1000 2800}%
\special{fp}%
\special{pa 1400 1200}%
\special{pa 1400 2000}%
\special{fp}%
\special{pa 1400 2000}%
\special{pa 1800 2000}%
\special{fp}%
\special{pa 1800 2000}%
\special{pa 1800 1200}%
\special{fp}%
\special{pa 1800 1200}%
\special{pa 1400 1200}%
\special{fp}%
}}%
%
{\color[named]{Black}{%
\special{pn 4}%
\special{pa 1400 2160}%
\special{pa 1220 2340}%
\special{fp}%
\special{pa 1400 2220}%
\special{pa 1220 2400}%
\special{fp}%
\special{pa 1400 2280}%
\special{pa 1280 2400}%
\special{fp}%
\special{pa 1400 2340}%
\special{pa 1340 2400}%
\special{fp}%
\special{pa 1400 2100}%
\special{pa 1200 2300}%
\special{fp}%
\special{pa 1400 2040}%
\special{pa 1230 2210}%
\special{fp}%
\special{pa 1380 2000}%
\special{pa 1210 2170}%
\special{fp}%
\special{pa 1320 2000}%
\special{pa 1200 2120}%
\special{fp}%
\special{pa 1260 2000}%
\special{pa 1200 2060}%
\special{fp}%
}}%
%
{\color[named]{Black}{%
\special{pn 4}%
\special{pa 1180 2200}%
\special{pa 1010 2370}%
\special{fp}%
\special{pa 1190 2250}%
\special{pa 1040 2400}%
\special{fp}%
\special{pa 1180 2320}%
\special{pa 1100 2400}%
\special{fp}%
\special{pa 1190 2370}%
\special{pa 1160 2400}%
\special{fp}%
\special{pa 1160 2160}%
\special{pa 1010 2310}%
\special{fp}%
\special{pa 1190 2070}%
\special{pa 1010 2250}%
\special{fp}%
\special{pa 1190 2010}%
\special{pa 1010 2190}%
\special{fp}%
\special{pa 1140 2000}%
\special{pa 1010 2130}%
\special{fp}%
\special{pa 1080 2000}%
\special{pa 1010 2070}%
\special{fp}%
\special{pa 1190 2130}%
\special{pa 1160 2160}%
\special{fp}%
}}%
%
{\color[named]{Black}{%
\special{pn 4}%
\special{pa 1390 1870}%
\special{pa 1260 2000}%
\special{fp}%
\special{pa 1390 1810}%
\special{pa 1210 1990}%
\special{fp}%
\special{pa 1390 1750}%
\special{pa 1200 1940}%
\special{fp}%
\special{pa 1390 1690}%
\special{pa 1200 1880}%
\special{fp}%
\special{pa 1390 1630}%
\special{pa 1200 1820}%
\special{fp}%
\special{pa 1360 1600}%
\special{pa 1200 1760}%
\special{fp}%
\special{pa 1300 1600}%
\special{pa 1200 1700}%
\special{fp}%
\special{pa 1240 1600}%
\special{pa 1200 1640}%
\special{fp}%
\special{pa 1390 1930}%
\special{pa 1320 2000}%
\special{fp}%
}}%
%
{\color[named]{Black}{%
\special{pn 4}%
\special{pa 1190 1830}%
\special{pa 1020 2000}%
\special{fp}%
\special{pa 1190 1770}%
\special{pa 1010 1950}%
\special{fp}%
\special{pa 1190 1710}%
\special{pa 1010 1890}%
\special{fp}%
\special{pa 1190 1650}%
\special{pa 1010 1830}%
\special{fp}%
\special{pa 1180 1600}%
\special{pa 1010 1770}%
\special{fp}%
\special{pa 1120 1600}%
\special{pa 1010 1710}%
\special{fp}%
\special{pa 1060 1600}%
\special{pa 1010 1650}%
\special{fp}%
\special{pa 1190 1890}%
\special{pa 1080 2000}%
\special{fp}%
\special{pa 1190 1950}%
\special{pa 1140 2000}%
\special{fp}%
}}%
%
{\color[named]{Black}{%
\special{pn 4}%
\special{pa 1390 1330}%
\special{pa 1200 1520}%
\special{fp}%
\special{pa 1390 1390}%
\special{pa 1200 1580}%
\special{fp}%
\special{pa 1390 1450}%
\special{pa 1240 1600}%
\special{fp}%
\special{pa 1390 1510}%
\special{pa 1300 1600}%
\special{fp}%
\special{pa 1390 1570}%
\special{pa 1360 1600}%
\special{fp}%
\special{pa 1390 1270}%
\special{pa 1200 1460}%
\special{fp}%
\special{pa 1390 1210}%
\special{pa 1200 1400}%
\special{fp}%
\special{pa 1220 1320}%
\special{pa 1200 1340}%
\special{fp}%
\special{pa 1340 1200}%
\special{pa 1240 1300}%
\special{fp}%
\special{pa 1280 1200}%
\special{pa 1240 1240}%
\special{fp}%
}}%
%
{\color[named]{Black}{%
\special{pn 4}%
\special{pa 1190 1410}%
\special{pa 1010 1590}%
\special{fp}%
\special{pa 1190 1470}%
\special{pa 1060 1600}%
\special{fp}%
\special{pa 1190 1530}%
\special{pa 1120 1600}%
\special{fp}%
\special{pa 1190 1350}%
\special{pa 1010 1530}%
\special{fp}%
\special{pa 1160 1320}%
\special{pa 1010 1470}%
\special{fp}%
\special{pa 1160 1260}%
\special{pa 1010 1410}%
\special{fp}%
\special{pa 1160 1200}%
\special{pa 1010 1350}%
\special{fp}%
\special{pa 1100 1200}%
\special{pa 1010 1290}%
\special{fp}%
\special{pa 1040 1200}%
\special{pa 1010 1230}%
\special{fp}%
}}%
%
{\color[named]{Black}{%
\special{pn 4}%
\special{pa 1360 1000}%
\special{pa 1200 1160}%
\special{fp}%
\special{pa 1400 1020}%
\special{pa 1220 1200}%
\special{fp}%
\special{pa 1400 1080}%
\special{pa 1280 1200}%
\special{fp}%
\special{pa 1400 1140}%
\special{pa 1340 1200}%
\special{fp}%
\special{pa 1300 1000}%
\special{pa 1200 1100}%
\special{fp}%
}}%
%
{\color[named]{Black}{%
\special{pn 4}%
\special{pa 1180 1000}%
\special{pa 1010 1170}%
\special{fp}%
\special{pa 1190 1050}%
\special{pa 1040 1200}%
\special{fp}%
\special{pa 1190 1110}%
\special{pa 1100 1200}%
\special{fp}%
\special{pa 1190 1170}%
\special{pa 1160 1200}%
\special{fp}%
\special{pa 1160 960}%
\special{pa 1010 1110}%
\special{fp}%
\special{pa 1250 810}%
\special{pa 1010 1050}%
\special{fp}%
\special{pa 1190 810}%
\special{pa 1010 990}%
\special{fp}%
\special{pa 1130 810}%
\special{pa 1010 930}%
\special{fp}%
\special{pa 1070 810}%
\special{pa 1010 870}%
\special{fp}%
\special{pa 1310 810}%
\special{pa 1160 960}%
\special{fp}%
\special{pa 1370 810}%
\special{pa 1200 980}%
\special{fp}%
\special{pa 1400 840}%
\special{pa 1250 990}%
\special{fp}%
\special{pa 1400 900}%
\special{pa 1310 990}%
\special{fp}%
\special{pa 1400 960}%
\special{pa 1370 990}%
\special{fp}%
}}%
%
{\color[named]{Black}{%
\special{pn 4}%
\special{pa 1730 810}%
\special{pa 1550 990}%
\special{fp}%
\special{pa 1670 810}%
\special{pa 1490 990}%
\special{fp}%
\special{pa 1610 810}%
\special{pa 1430 990}%
\special{fp}%
\special{pa 1550 810}%
\special{pa 1400 960}%
\special{fp}%
\special{pa 1490 810}%
\special{pa 1400 900}%
\special{fp}%
\special{pa 1430 810}%
\special{pa 1400 840}%
\special{fp}%
\special{pa 1790 810}%
\special{pa 1610 990}%
\special{fp}%
\special{pa 1800 860}%
\special{pa 1670 990}%
\special{fp}%
\special{pa 1800 920}%
\special{pa 1730 990}%
\special{fp}%
}}%
%
{\color[named]{Black}{%
\special{pn 4}%
\special{pa 1600 1000}%
\special{pa 1410 1190}%
\special{fp}%
\special{pa 1540 1000}%
\special{pa 1400 1140}%
\special{fp}%
\special{pa 1480 1000}%
\special{pa 1400 1080}%
\special{fp}%
\special{pa 1660 1000}%
\special{pa 1470 1190}%
\special{fp}%
\special{pa 1720 1000}%
\special{pa 1530 1190}%
\special{fp}%
\special{pa 1780 1000}%
\special{pa 1590 1190}%
\special{fp}%
\special{pa 1800 1040}%
\special{pa 1650 1190}%
\special{fp}%
\special{pa 1800 1100}%
\special{pa 1710 1190}%
\special{fp}%
\special{pa 1800 1160}%
\special{pa 1770 1190}%
\special{fp}%
}}%
%
{\color[named]{Black}{%
\special{pn 4}%
\special{pa 2190 890}%
\special{pa 2040 1040}%
\special{fp}%
\special{pa 2190 950}%
\special{pa 2000 1140}%
\special{fp}%
\special{pa 2190 1010}%
\special{pa 2010 1190}%
\special{fp}%
\special{pa 2190 1070}%
\special{pa 2060 1200}%
\special{fp}%
\special{pa 2190 1130}%
\special{pa 2120 1200}%
\special{fp}%
\special{pa 2040 1040}%
\special{pa 2000 1080}%
\special{fp}%
\special{pa 2190 830}%
\special{pa 2020 1000}%
\special{fp}%
\special{pa 2150 810}%
\special{pa 1990 970}%
\special{fp}%
\special{pa 2090 810}%
\special{pa 1910 990}%
\special{fp}%
\special{pa 2030 810}%
\special{pa 1850 990}%
\special{fp}%
\special{pa 1970 810}%
\special{pa 1800 980}%
\special{fp}%
\special{pa 1910 810}%
\special{pa 1800 920}%
\special{fp}%
\special{pa 1850 810}%
\special{pa 1800 860}%
\special{fp}%
}}%
%
{\color[named]{Black}{%
\special{pn 4}%
\special{pa 1980 1040}%
\special{pa 1820 1200}%
\special{fp}%
\special{pa 1960 1000}%
\special{pa 1800 1160}%
\special{fp}%
\special{pa 1900 1000}%
\special{pa 1800 1100}%
\special{fp}%
\special{pa 1840 1000}%
\special{pa 1800 1040}%
\special{fp}%
\special{pa 1990 1090}%
\special{pa 1880 1200}%
\special{fp}%
\special{pa 1990 1150}%
\special{pa 1940 1200}%
\special{fp}%
}}%
%
{\color[named]{Black}{%
\special{pn 4}%
\special{pa 1980 1400}%
\special{pa 1810 1570}%
\special{fp}%
\special{pa 2010 1430}%
\special{pa 1840 1600}%
\special{fp}%
\special{pa 2100 1400}%
\special{pa 1900 1600}%
\special{fp}%
\special{pa 2120 1440}%
\special{pa 1960 1600}%
\special{fp}%
\special{pa 2120 1500}%
\special{pa 2020 1600}%
\special{fp}%
\special{pa 2120 1560}%
\special{pa 2080 1600}%
\special{fp}%
\special{pa 1960 1360}%
\special{pa 1810 1510}%
\special{fp}%
\special{pa 1990 1270}%
\special{pa 1810 1450}%
\special{fp}%
\special{pa 1990 1210}%
\special{pa 1810 1390}%
\special{fp}%
\special{pa 1940 1200}%
\special{pa 1810 1330}%
\special{fp}%
\special{pa 1880 1200}%
\special{pa 1810 1270}%
\special{fp}%
\special{pa 1990 1330}%
\special{pa 1960 1360}%
\special{fp}%
}}%
%
{\color[named]{Black}{%
\special{pn 4}%
\special{pa 2190 1250}%
\special{pa 2050 1390}%
\special{fp}%
\special{pa 2180 1200}%
\special{pa 2010 1370}%
\special{fp}%
\special{pa 2120 1200}%
\special{pa 2000 1320}%
\special{fp}%
\special{pa 2060 1200}%
\special{pa 2000 1260}%
\special{fp}%
\special{pa 2190 1310}%
\special{pa 2120 1380}%
\special{fp}%
\special{pa 2190 1370}%
\special{pa 2150 1410}%
\special{fp}%
\special{pa 2190 1430}%
\special{pa 2130 1490}%
\special{fp}%
\special{pa 2190 1490}%
\special{pa 2130 1550}%
\special{fp}%
\special{pa 2190 1550}%
\special{pa 2140 1600}%
\special{fp}%
}}%
%
{\color[named]{Black}{%
\special{pn 4}%
\special{pa 2120 1800}%
\special{pa 1920 2000}%
\special{fp}%
\special{pa 2120 1740}%
\special{pa 1860 2000}%
\special{fp}%
\special{pa 2120 1680}%
\special{pa 1810 1990}%
\special{fp}%
\special{pa 2120 1620}%
\special{pa 1810 1930}%
\special{fp}%
\special{pa 2080 1600}%
\special{pa 1810 1870}%
\special{fp}%
\special{pa 2020 1600}%
\special{pa 1810 1810}%
\special{fp}%
\special{pa 1960 1600}%
\special{pa 1810 1750}%
\special{fp}%
\special{pa 1900 1600}%
\special{pa 1810 1690}%
\special{fp}%
\special{pa 1840 1600}%
\special{pa 1810 1630}%
\special{fp}%
\special{pa 2120 1860}%
\special{pa 1980 2000}%
\special{fp}%
\special{pa 2120 1920}%
\special{pa 2040 2000}%
\special{fp}%
}}%
%
{\color[named]{Black}{%
\special{pn 4}%
\special{pa 2190 1910}%
\special{pa 2130 1970}%
\special{fp}%
\special{pa 2190 1850}%
\special{pa 2130 1910}%
\special{fp}%
\special{pa 2190 1790}%
\special{pa 2130 1850}%
\special{fp}%
\special{pa 2190 1730}%
\special{pa 2130 1790}%
\special{fp}%
\special{pa 2190 1670}%
\special{pa 2130 1730}%
\special{fp}%
\special{pa 2190 1610}%
\special{pa 2130 1670}%
\special{fp}%
\special{pa 2190 1970}%
\special{pa 2160 2000}%
\special{fp}%
}}%
%
{\color[named]{Black}{%
\special{pn 4}%
\special{pa 1520 2220}%
\special{pa 1400 2340}%
\special{fp}%
\special{pa 1520 2280}%
\special{pa 1410 2390}%
\special{fp}%
\special{pa 1520 2340}%
\special{pa 1460 2400}%
\special{fp}%
\special{pa 1520 2160}%
\special{pa 1400 2280}%
\special{fp}%
\special{pa 1490 2130}%
\special{pa 1400 2220}%
\special{fp}%
\special{pa 1490 2070}%
\special{pa 1400 2160}%
\special{fp}%
\special{pa 1490 2010}%
\special{pa 1400 2100}%
\special{fp}%
\special{pa 1430 2010}%
\special{pa 1400 2040}%
\special{fp}%
\special{pa 1550 2010}%
\special{pa 1500 2060}%
\special{fp}%
\special{pa 1610 2010}%
\special{pa 1560 2060}%
\special{fp}%
\special{pa 1670 2010}%
\special{pa 1590 2090}%
\special{fp}%
\special{pa 1730 2010}%
\special{pa 1650 2090}%
\special{fp}%
\special{pa 1790 2010}%
\special{pa 1710 2090}%
\special{fp}%
\special{pa 1800 2060}%
\special{pa 1770 2090}%
\special{fp}%
}}%
%
{\color[named]{Black}{%
\special{pn 4}%
\special{pa 1920 2000}%
\special{pa 1830 2090}%
\special{fp}%
\special{pa 1860 2000}%
\special{pa 1800 2060}%
\special{fp}%
\special{pa 1980 2000}%
\special{pa 1890 2090}%
\special{fp}%
\special{pa 2040 2000}%
\special{pa 1950 2090}%
\special{fp}%
\special{pa 2100 2000}%
\special{pa 2010 2090}%
\special{fp}%
\special{pa 2120 2040}%
\special{pa 2100 2060}%
\special{fp}%
}}%
%
{\color[named]{Black}{%
\special{pn 4}%
\special{pa 1800 2240}%
\special{pa 1640 2400}%
\special{fp}%
\special{pa 1800 2180}%
\special{pa 1580 2400}%
\special{fp}%
\special{pa 1800 2120}%
\special{pa 1530 2390}%
\special{fp}%
\special{pa 1760 2100}%
\special{pa 1530 2330}%
\special{fp}%
\special{pa 1700 2100}%
\special{pa 1530 2270}%
\special{fp}%
\special{pa 1640 2100}%
\special{pa 1530 2210}%
\special{fp}%
\special{pa 1550 2130}%
\special{pa 1530 2150}%
\special{fp}%
\special{pa 1800 2300}%
\special{pa 1700 2400}%
\special{fp}%
\special{pa 1800 2360}%
\special{pa 1760 2400}%
\special{fp}%
}}%
%
{\color[named]{Black}{%
\special{pn 4}%
\special{pa 2060 2100}%
\special{pa 1800 2360}%
\special{fp}%
\special{pa 2090 2130}%
\special{pa 1820 2400}%
\special{fp}%
\special{pa 2150 2130}%
\special{pa 1880 2400}%
\special{fp}%
\special{pa 2190 2150}%
\special{pa 1940 2400}%
\special{fp}%
\special{pa 2190 2210}%
\special{pa 2000 2400}%
\special{fp}%
\special{pa 2190 2270}%
\special{pa 2060 2400}%
\special{fp}%
\special{pa 2190 2330}%
\special{pa 2120 2400}%
\special{fp}%
\special{pa 2190 2090}%
\special{pa 2160 2120}%
\special{fp}%
\special{pa 2190 2030}%
\special{pa 2160 2060}%
\special{fp}%
\special{pa 2160 2000}%
\special{pa 2130 2030}%
\special{fp}%
\special{pa 2000 2100}%
\special{pa 1800 2300}%
\special{fp}%
\special{pa 1940 2100}%
\special{pa 1800 2240}%
\special{fp}%
\special{pa 1880 2100}%
\special{pa 1800 2180}%
\special{fp}%
}}%
%
{\color[named]{Black}{%
\special{pn 4}%
\special{pa 1720 2500}%
\special{pa 1430 2790}%
\special{fp}%
\special{pa 1660 2500}%
\special{pa 1400 2760}%
\special{fp}%
\special{pa 1560 2540}%
\special{pa 1400 2700}%
\special{fp}%
\special{pa 1500 2540}%
\special{pa 1400 2640}%
\special{fp}%
\special{pa 1500 2480}%
\special{pa 1400 2580}%
\special{fp}%
\special{pa 1510 2410}%
\special{pa 1400 2520}%
\special{fp}%
\special{pa 1460 2400}%
\special{pa 1400 2460}%
\special{fp}%
\special{pa 1600 2500}%
\special{pa 1570 2530}%
\special{fp}%
\special{pa 1780 2500}%
\special{pa 1490 2790}%
\special{fp}%
\special{pa 1800 2540}%
\special{pa 1550 2790}%
\special{fp}%
\special{pa 1800 2600}%
\special{pa 1610 2790}%
\special{fp}%
\special{pa 1800 2660}%
\special{pa 1670 2790}%
\special{fp}%
\special{pa 1800 2720}%
\special{pa 1730 2790}%
\special{fp}%
}}%
%
{\color[named]{Black}{%
\special{pn 4}%
\special{pa 1760 2400}%
\special{pa 1670 2490}%
\special{fp}%
\special{pa 1700 2400}%
\special{pa 1610 2490}%
\special{fp}%
\special{pa 1640 2400}%
\special{pa 1570 2470}%
\special{fp}%
\special{pa 1580 2400}%
\special{pa 1530 2450}%
\special{fp}%
\special{pa 1800 2420}%
\special{pa 1730 2490}%
\special{fp}%
}}%
%
{\color[named]{Black}{%
\special{pn 4}%
\special{pa 2180 2400}%
\special{pa 2090 2490}%
\special{fp}%
\special{pa 2120 2400}%
\special{pa 2030 2490}%
\special{fp}%
\special{pa 2060 2400}%
\special{pa 1970 2490}%
\special{fp}%
\special{pa 2000 2400}%
\special{pa 1910 2490}%
\special{fp}%
\special{pa 1940 2400}%
\special{pa 1850 2490}%
\special{fp}%
\special{pa 1880 2400}%
\special{pa 1800 2480}%
\special{fp}%
\special{pa 2200 2440}%
\special{pa 2150 2490}%
\special{fp}%
}}%
%
{\color[named]{Black}{%
\special{pn 4}%
\special{pa 2190 2510}%
\special{pa 1910 2790}%
\special{fp}%
\special{pa 2140 2500}%
\special{pa 1850 2790}%
\special{fp}%
\special{pa 2080 2500}%
\special{pa 1800 2780}%
\special{fp}%
\special{pa 2020 2500}%
\special{pa 1800 2720}%
\special{fp}%
\special{pa 1960 2500}%
\special{pa 1800 2660}%
\special{fp}%
\special{pa 1900 2500}%
\special{pa 1800 2600}%
\special{fp}%
\special{pa 1840 2500}%
\special{pa 1800 2540}%
\special{fp}%
\special{pa 2200 2560}%
\special{pa 1970 2790}%
\special{fp}%
\special{pa 2200 2620}%
\special{pa 2030 2790}%
\special{fp}%
\special{pa 2200 2680}%
\special{pa 2090 2790}%
\special{fp}%
\special{pa 2200 2740}%
\special{pa 2150 2790}%
\special{fp}%
}}%
%
{\color[named]{Black}{%
\special{pn 4}%
\special{pa 2600 2520}%
\special{pa 2330 2790}%
\special{fp}%
\special{pa 2560 2500}%
\special{pa 2270 2790}%
\special{fp}%
\special{pa 2500 2500}%
\special{pa 2210 2790}%
\special{fp}%
\special{pa 2440 2500}%
\special{pa 2200 2740}%
\special{fp}%
\special{pa 2380 2500}%
\special{pa 2200 2680}%
\special{fp}%
\special{pa 2320 2500}%
\special{pa 2200 2620}%
\special{fp}%
\special{pa 2260 2500}%
\special{pa 2200 2560}%
\special{fp}%
\special{pa 2600 2580}%
\special{pa 2390 2790}%
\special{fp}%
\special{pa 2600 2640}%
\special{pa 2450 2790}%
\special{fp}%
\special{pa 2600 2700}%
\special{pa 2510 2790}%
\special{fp}%
\special{pa 2600 2760}%
\special{pa 2570 2790}%
\special{fp}%
}}%
%
{\color[named]{Black}{%
\special{pn 4}%
\special{pa 2470 2410}%
\special{pa 2390 2490}%
\special{fp}%
\special{pa 2410 2410}%
\special{pa 2330 2490}%
\special{fp}%
\special{pa 2350 2410}%
\special{pa 2270 2490}%
\special{fp}%
\special{pa 2290 2410}%
\special{pa 2210 2490}%
\special{fp}%
\special{pa 2230 2410}%
\special{pa 2200 2440}%
\special{fp}%
\special{pa 2530 2410}%
\special{pa 2450 2490}%
\special{fp}%
\special{pa 2590 2410}%
\special{pa 2510 2490}%
\special{fp}%
\special{pa 2600 2460}%
\special{pa 2570 2490}%
\special{fp}%
}}%
%
{\color[named]{Black}{%
\special{pn 4}%
\special{pa 2990 2610}%
\special{pa 2810 2790}%
\special{fp}%
\special{pa 2990 2550}%
\special{pa 2750 2790}%
\special{fp}%
\special{pa 2990 2490}%
\special{pa 2690 2790}%
\special{fp}%
\special{pa 2990 2430}%
\special{pa 2630 2790}%
\special{fp}%
\special{pa 2960 2400}%
\special{pa 2600 2760}%
\special{fp}%
\special{pa 2760 2540}%
\special{pa 2600 2700}%
\special{fp}%
\special{pa 2700 2540}%
\special{pa 2600 2640}%
\special{fp}%
\special{pa 2680 2500}%
\special{pa 2600 2580}%
\special{fp}%
\special{pa 2900 2400}%
\special{pa 2770 2530}%
\special{fp}%
\special{pa 2840 2400}%
\special{pa 2770 2470}%
\special{fp}%
\special{pa 2780 2400}%
\special{pa 2730 2450}%
\special{fp}%
\special{pa 2990 2670}%
\special{pa 2870 2790}%
\special{fp}%
\special{pa 2990 2730}%
\special{pa 2930 2790}%
\special{fp}%
}}%
%
{\color[named]{Black}{%
\special{pn 4}%
\special{pa 2710 2410}%
\special{pa 2630 2490}%
\special{fp}%
\special{pa 2660 2400}%
\special{pa 2600 2460}%
\special{fp}%
}}%
%
{\color[named]{Black}{%
\special{pn 4}%
\special{pa 2990 2070}%
\special{pa 2730 2330}%
\special{fp}%
\special{pa 2990 2130}%
\special{pa 2730 2390}%
\special{fp}%
\special{pa 2990 2190}%
\special{pa 2780 2400}%
\special{fp}%
\special{pa 2990 2250}%
\special{pa 2840 2400}%
\special{fp}%
\special{pa 2990 2310}%
\special{pa 2900 2400}%
\special{fp}%
\special{pa 2990 2370}%
\special{pa 2960 2400}%
\special{fp}%
\special{pa 2990 2010}%
\special{pa 2730 2270}%
\special{fp}%
\special{pa 2940 2000}%
\special{pa 2730 2210}%
\special{fp}%
\special{pa 2880 2000}%
\special{pa 2730 2150}%
\special{fp}%
\special{pa 2820 2000}%
\special{pa 2730 2090}%
\special{fp}%
\special{pa 2760 2000}%
\special{pa 2730 2030}%
\special{fp}%
}}%
%
{\color[named]{Black}{%
\special{pn 4}%
\special{pa 2720 2340}%
\special{pa 2660 2400}%
\special{fp}%
\special{pa 2720 2280}%
\special{pa 2610 2390}%
\special{fp}%
\special{pa 2720 2220}%
\special{pa 2610 2330}%
\special{fp}%
\special{pa 2720 2160}%
\special{pa 2610 2270}%
\special{fp}%
\special{pa 2720 2100}%
\special{pa 2610 2210}%
\special{fp}%
\special{pa 2720 2040}%
\special{pa 2610 2150}%
\special{fp}%
\special{pa 2700 2000}%
\special{pa 2610 2090}%
\special{fp}%
\special{pa 2640 2000}%
\special{pa 2610 2030}%
\special{fp}%
}}%
%
{\color[named]{Black}{%
\special{pn 4}%
\special{pa 2850 1610}%
\special{pa 2690 1770}%
\special{fp}%
\special{pa 2790 1610}%
\special{pa 2610 1790}%
\special{fp}%
\special{pa 2690 1770}%
\special{pa 2610 1850}%
\special{fp}%
\special{pa 2710 1810}%
\special{pa 2610 1910}%
\special{fp}%
\special{pa 2720 1860}%
\special{pa 2610 1970}%
\special{fp}%
\special{pa 2720 1920}%
\special{pa 2640 2000}%
\special{fp}%
\special{pa 2730 1610}%
\special{pa 2610 1730}%
\special{fp}%
\special{pa 2670 1610}%
\special{pa 2610 1670}%
\special{fp}%
\special{pa 2910 1610}%
\special{pa 2730 1790}%
\special{fp}%
\special{pa 2970 1610}%
\special{pa 2780 1800}%
\special{fp}%
\special{pa 3000 1640}%
\special{pa 2840 1800}%
\special{fp}%
\special{pa 3000 1700}%
\special{pa 2920 1780}%
\special{fp}%
\special{pa 3000 1760}%
\special{pa 2980 1780}%
\special{fp}%
}}%
%
{\color[named]{Black}{%
\special{pn 4}%
\special{pa 3000 1880}%
\special{pa 2880 2000}%
\special{fp}%
\special{pa 2970 1850}%
\special{pa 2820 2000}%
\special{fp}%
\special{pa 2910 1850}%
\special{pa 2760 2000}%
\special{fp}%
\special{pa 2890 1810}%
\special{pa 2730 1970}%
\special{fp}%
\special{pa 2830 1810}%
\special{pa 2730 1910}%
\special{fp}%
\special{pa 3000 1940}%
\special{pa 2940 2000}%
\special{fp}%
\special{pa 3000 1820}%
\special{pa 2980 1840}%
\special{fp}%
}}%
%
{\color[named]{Black}{%
\special{pn 4}%
\special{pa 3280 1660}%
\special{pa 3140 1800}%
\special{fp}%
\special{pa 3270 1610}%
\special{pa 3080 1800}%
\special{fp}%
\special{pa 3220 1600}%
\special{pa 3020 1800}%
\special{fp}%
\special{pa 3160 1600}%
\special{pa 3000 1760}%
\special{fp}%
\special{pa 3100 1600}%
\special{pa 3000 1700}%
\special{fp}%
\special{pa 3040 1600}%
\special{pa 3000 1640}%
\special{fp}%
\special{pa 3280 1720}%
\special{pa 3200 1800}%
\special{fp}%
}}%
%
{\color[named]{Black}{%
\special{pn 4}%
\special{pa 3290 1830}%
\special{pa 3130 1990}%
\special{fp}%
\special{pa 3250 1810}%
\special{pa 3070 1990}%
\special{fp}%
\special{pa 3190 1810}%
\special{pa 3010 1990}%
\special{fp}%
\special{pa 3130 1810}%
\special{pa 3000 1940}%
\special{fp}%
\special{pa 3070 1810}%
\special{pa 3000 1880}%
\special{fp}%
\special{pa 3330 1850}%
\special{pa 3190 1990}%
\special{fp}%
\special{pa 3390 1850}%
\special{pa 3250 1990}%
\special{fp}%
\special{pa 3390 1910}%
\special{pa 3310 1990}%
\special{fp}%
\special{pa 3390 1790}%
\special{pa 3330 1850}%
\special{fp}%
\special{pa 3390 1730}%
\special{pa 3310 1810}%
\special{fp}%
\special{pa 3390 1670}%
\special{pa 3290 1770}%
\special{fp}%
\special{pa 3390 1610}%
\special{pa 3290 1710}%
\special{fp}%
\special{pa 3340 1600}%
\special{pa 3290 1650}%
\special{fp}%
}}%
%
{\color[named]{Black}{%
\special{pn 4}%
\special{pa 3280 1300}%
\special{pa 3010 1570}%
\special{fp}%
\special{pa 3280 1360}%
\special{pa 3040 1600}%
\special{fp}%
\special{pa 3280 1420}%
\special{pa 3100 1600}%
\special{fp}%
\special{pa 3280 1480}%
\special{pa 3260 1500}%
\special{fp}%
\special{pa 3250 1510}%
\special{pa 3160 1600}%
\special{fp}%
\special{pa 3250 1570}%
\special{pa 3220 1600}%
\special{fp}%
\special{pa 3280 1240}%
\special{pa 3010 1510}%
\special{fp}%
\special{pa 3260 1200}%
\special{pa 3010 1450}%
\special{fp}%
\special{pa 3200 1200}%
\special{pa 3010 1390}%
\special{fp}%
\special{pa 3140 1200}%
\special{pa 3010 1330}%
\special{fp}%
\special{pa 3080 1200}%
\special{pa 3010 1270}%
\special{fp}%
}}%
%
{\color[named]{Black}{%
\special{pn 4}%
\special{pa 3390 1310}%
\special{pa 3290 1410}%
\special{fp}%
\special{pa 3390 1370}%
\special{pa 3290 1470}%
\special{fp}%
\special{pa 3390 1430}%
\special{pa 3320 1500}%
\special{fp}%
\special{pa 3390 1490}%
\special{pa 3320 1560}%
\special{fp}%
\special{pa 3390 1550}%
\special{pa 3340 1600}%
\special{fp}%
\special{pa 3390 1250}%
\special{pa 3290 1350}%
\special{fp}%
\special{pa 3380 1200}%
\special{pa 3290 1290}%
\special{fp}%
\special{pa 3320 1200}%
\special{pa 3290 1230}%
\special{fp}%
}}%
%
{\color[named]{Black}{%
\special{pn 4}%
\special{pa 3280 940}%
\special{pa 3020 1200}%
\special{fp}%
\special{pa 3270 890}%
\special{pa 3010 1150}%
\special{fp}%
\special{pa 3290 810}%
\special{pa 3010 1090}%
\special{fp}%
\special{pa 3350 810}%
\special{pa 3280 880}%
\special{fp}%
\special{pa 3390 830}%
\special{pa 3310 910}%
\special{fp}%
\special{pa 3390 890}%
\special{pa 3290 990}%
\special{fp}%
\special{pa 3390 950}%
\special{pa 3290 1050}%
\special{fp}%
\special{pa 3390 1010}%
\special{pa 3290 1110}%
\special{fp}%
\special{pa 3390 1070}%
\special{pa 3290 1170}%
\special{fp}%
\special{pa 3390 1130}%
\special{pa 3320 1200}%
\special{fp}%
\special{pa 3230 810}%
\special{pa 3010 1030}%
\special{fp}%
\special{pa 3170 810}%
\special{pa 3010 970}%
\special{fp}%
\special{pa 3110 810}%
\special{pa 3010 910}%
\special{fp}%
\special{pa 3050 810}%
\special{pa 3010 850}%
\special{fp}%
\special{pa 3280 1000}%
\special{pa 3080 1200}%
\special{fp}%
\special{pa 3280 1060}%
\special{pa 3140 1200}%
\special{fp}%
\special{pa 3280 1120}%
\special{pa 3200 1200}%
\special{fp}%
}}%
%
{\color[named]{Black}{%
\special{pn 4}%
\special{pa 1160 2560}%
\special{pa 1010 2710}%
\special{fp}%
\special{pa 1180 2600}%
\special{pa 1010 2770}%
\special{fp}%
\special{pa 1210 2630}%
\special{pa 1050 2790}%
\special{fp}%
\special{pa 1400 2500}%
\special{pa 1110 2790}%
\special{fp}%
\special{pa 1400 2440}%
\special{pa 1230 2610}%
\special{fp}%
\special{pa 1380 2400}%
\special{pa 1210 2570}%
\special{fp}%
\special{pa 1320 2400}%
\special{pa 1200 2520}%
\special{fp}%
\special{pa 1260 2400}%
\special{pa 1200 2460}%
\special{fp}%
\special{pa 1400 2560}%
\special{pa 1170 2790}%
\special{fp}%
\special{pa 1400 2620}%
\special{pa 1230 2790}%
\special{fp}%
\special{pa 1400 2680}%
\special{pa 1290 2790}%
\special{fp}%
\special{pa 1400 2740}%
\special{pa 1350 2790}%
\special{fp}%
\special{pa 1190 2470}%
\special{pa 1010 2650}%
\special{fp}%
\special{pa 1190 2410}%
\special{pa 1010 2590}%
\special{fp}%
\special{pa 1140 2400}%
\special{pa 1010 2530}%
\special{fp}%
\special{pa 1080 2400}%
\special{pa 1010 2470}%
\special{fp}%
\special{pa 1190 2530}%
\special{pa 1160 2560}%
\special{fp}%
}}%
%
{\color[named]{Black}{%
\special{pn 8}%
\special{pa 3080 2990}%
\special{pa 2880 2620}%
\special{fp}%
\special{sh 1}%
\special{pa 2880 2620}%
\special{pa 2894 2688}%
\special{pa 2906 2668}%
\special{pa 2930 2670}%
\special{pa 2880 2620}%
\special{fp}%
}}%
\end{picture}%

%% file: graph_2.tex
\unitlength 0.1in
\begin{picture}( 40.6100, 29.6000)(  0.9000,-33.0000)
%
{\color[named]{Black}{%
\special{pn 8}%
\special{pa 800 406}%
\special{pa 4000 406}%
\special{pa 4000 3206}%
\special{pa 800 3206}%
\special{pa 800 406}%
\special{pa 4000 406}%
\special{fp}%
}}%
%
{\color[named]{Black}{%
\special{pn 8}%
\special{pa 4000 2606}%
\special{pa 800 2606}%
\special{fp}%
\special{pa 800 2406}%
\special{pa 4000 2406}%
\special{fp}%
\special{pa 4000 2206}%
\special{pa 800 2206}%
\special{fp}%
\special{pa 800 2006}%
\special{pa 4000 2006}%
\special{fp}%
\special{pa 4000 1806}%
\special{pa 800 1806}%
\special{fp}%
\special{pa 800 1606}%
\special{pa 4000 1606}%
\special{fp}%
}}%
%
{\color[named]{Black}{%
\special{pn 13}%
\special{pa 3160 1560}%
\special{pa 3240 1640}%
\special{fp}%
\special{pa 3240 1560}%
\special{pa 3160 1650}%
\special{fp}%
}}%
%
{\color[named]{Black}{%
\special{pn 8}%
\special{pa 2400 2606}%
\special{pa 2400 2406}%
\special{dt 0.045}%
\special{pa 1200 2606}%
\special{pa 1200 1606}%
\special{dt 0.045}%
\special{pa 1800 1806}%
\special{pa 1800 2006}%
\special{dt 0.045}%
\special{pa 3600 2406}%
\special{pa 3600 2006}%
\special{dt 0.045}%
\special{pa 2800 2206}%
\special{pa 2800 1606}%
\special{dt 0.045}%
}}%
%
{\color[named]{Black}{%
\special{pn 13}%
\special{pa 1360 766}%
\special{pa 1440 846}%
\special{fp}%
\special{pa 1440 766}%
\special{pa 1360 856}%
\special{fp}%
}}%
%
{\color[named]{Black}{%
\special{pn 13}%
\special{pa 800 806}%
\special{pa 4000 806}%
\special{fp}%
}}%
%
{\color[named]{Black}{%
\special{pn 13}%
\special{pa 1960 766}%
\special{pa 2040 846}%
\special{fp}%
\special{pa 2040 766}%
\special{pa 1960 856}%
\special{fp}%
}}%
%
{\color[named]{Black}{%
\special{pn 13}%
\special{pa 2160 766}%
\special{pa 2240 846}%
\special{fp}%
\special{pa 2240 766}%
\special{pa 2160 856}%
\special{fp}%
}}%
%
{\color[named]{Black}{%
\special{pn 13}%
\special{pa 2960 766}%
\special{pa 3040 846}%
\special{fp}%
\special{pa 3040 766}%
\special{pa 2960 856}%
\special{fp}%
}}%
%
{\color[named]{Black}{%
\special{pn 13}%
\special{pa 3360 766}%
\special{pa 3440 846}%
\special{fp}%
\special{pa 3440 766}%
\special{pa 3360 856}%
\special{fp}%
}}%
%
{\color[named]{Black}{%
\special{pn 13}%
\special{pa 2560 766}%
\special{pa 2640 846}%
\special{fp}%
\special{pa 2640 766}%
\special{pa 2560 856}%
\special{fp}%
}}%
%
{\color[named]{Black}{%
\special{pn 13}%
\special{pa 2000 396}%
\special{pa 2000 3196}%
\special{fp}%
\special{pa 2200 3196}%
\special{pa 2200 396}%
\special{fp}%
\special{pa 2600 396}%
\special{pa 2600 3196}%
\special{fp}%
\special{pa 3000 3196}%
\special{pa 3000 396}%
\special{fp}%
\special{pa 3400 396}%
\special{pa 3400 3196}%
\special{fp}%
\special{pa 1400 3196}%
\special{pa 1400 396}%
\special{fp}%
}}%
\put(32.0000,-14.2500){\makebox(0,0){\mbox{\Large $x$}}}%
\put(14.8000,-9.1500){\makebox(0,0)[lt]{\mbox{\Large $y$}}}%
\put(43.1000,-22.0500){\makebox(0,0){\mbox{\Large $L$}}}%
%
{\color[named]{Black}{%
\special{pn 8}%
\special{pa 4030 1806}%
\special{pa 4060 1818}%
\special{pa 4080 1842}%
\special{pa 4096 1870}%
\special{pa 4108 1898}%
\special{pa 4118 1930}%
\special{pa 4126 1960}%
\special{pa 4132 1992}%
\special{pa 4138 2024}%
\special{pa 4142 2054}%
\special{pa 4146 2086}%
\special{pa 4150 2150}%
\special{pa 4152 2182}%
\special{pa 4152 2246}%
\special{pa 4148 2310}%
\special{pa 4144 2342}%
\special{pa 4140 2374}%
\special{pa 4136 2406}%
\special{pa 4130 2436}%
\special{pa 4122 2468}%
\special{pa 4114 2498}%
\special{pa 4102 2528}%
\special{pa 4088 2558}%
\special{pa 4070 2584}%
\special{pa 4044 2604}%
\special{pa 4030 2606}%
\special{sp}%
}}%
%
{\color[named]{Black}{%
\special{pn 13}%
\special{pa 2360 2566}%
\special{pa 2440 2646}%
\special{fp}%
\special{pa 2440 2566}%
\special{pa 2360 2656}%
\special{fp}%
}}%
%
{\color[named]{Black}{%
\special{pn 13}%
\special{pa 2760 2166}%
\special{pa 2840 2246}%
\special{fp}%
\special{pa 2840 2166}%
\special{pa 2760 2256}%
\special{fp}%
}}%
%
{\color[named]{Black}{%
\special{pn 13}%
\special{pa 3560 2366}%
\special{pa 3640 2446}%
\special{fp}%
\special{pa 3640 2366}%
\special{pa 3560 2456}%
\special{fp}%
}}%
%
{\color[named]{Black}{%
\special{pn 13}%
\special{pa 3560 1966}%
\special{pa 3640 2046}%
\special{fp}%
\special{pa 3640 1966}%
\special{pa 3560 2056}%
\special{fp}%
}}%
%
{\color[named]{Black}{%
\special{pn 13}%
\special{pa 2760 1566}%
\special{pa 2840 1646}%
\special{fp}%
\special{pa 2840 1566}%
\special{pa 2760 1656}%
\special{fp}%
}}%
%
{\color[named]{Black}{%
\special{pn 13}%
\special{pa 1760 1766}%
\special{pa 1840 1846}%
\special{fp}%
\special{pa 1840 1766}%
\special{pa 1760 1856}%
\special{fp}%
}}%
%
{\color[named]{Black}{%
\special{pn 13}%
\special{pa 1760 1966}%
\special{pa 1840 2046}%
\special{fp}%
\special{pa 1840 1966}%
\special{pa 1760 2056}%
\special{fp}%
}}%
%
{\color[named]{Black}{%
\special{pn 13}%
\special{pa 1160 2566}%
\special{pa 1240 2646}%
\special{fp}%
\special{pa 1240 2566}%
\special{pa 1160 2656}%
\special{fp}%
}}%
%
{\color[named]{Black}{%
\special{pn 13}%
\special{pa 1160 1566}%
\special{pa 1240 1646}%
\special{fp}%
\special{pa 1240 1566}%
\special{pa 1160 1656}%
\special{fp}%
}}%
%
{\color[named]{Black}{%
\special{pn 13}%
\special{pa 2360 2366}%
\special{pa 2440 2446}%
\special{fp}%
\special{pa 2440 2366}%
\special{pa 2360 2456}%
\special{fp}%
}}%
%
{\color[named]{Black}{%
\special{pn 13}%
\special{pa 2360 1966}%
\special{pa 2440 2046}%
\special{fp}%
\special{pa 2440 1966}%
\special{pa 2360 2056}%
\special{fp}%
}}%
%
{\color[named]{Black}{%
\special{pn 13}%
\special{pa 1560 2366}%
\special{pa 1640 2446}%
\special{fp}%
\special{pa 1640 2366}%
\special{pa 1560 2456}%
\special{fp}%
}}%
%
{\color[named]{Black}{%
\special{pn 13}%
\special{pa 1560 1566}%
\special{pa 1640 1646}%
\special{fp}%
\special{pa 1640 1566}%
\special{pa 1560 1656}%
\special{fp}%
}}%
%
{\color[named]{Black}{%
\special{pn 13}%
\special{pa 2360 1766}%
\special{pa 2440 1846}%
\special{fp}%
\special{pa 2440 1766}%
\special{pa 2360 1856}%
\special{fp}%
}}%
%
{\color[named]{Black}{%
\special{pn 8}%
\special{pa 2400 1806}%
\special{pa 2400 2406}%
\special{dt 0.045}%
\special{pa 1600 2406}%
\special{pa 1600 1606}%
\special{dt 0.045}%
}}%
%
{\color[named]{Black}{%
\special{pn 13}%
\special{pa 1160 2166}%
\special{pa 1240 2246}%
\special{fp}%
\special{pa 1240 2166}%
\special{pa 1160 2256}%
\special{fp}%
}}%
%
{\color[named]{Black}{%
\special{pn 8}%
\special{pa 1200 406}%
\special{pa 1200 1606}%
\special{dt 0.045}%
\special{pa 1200 2606}%
\special{pa 1200 3206}%
\special{dt 0.045}%
\special{pa 1600 3206}%
\special{pa 1600 2406}%
\special{dt 0.045}%
\special{pa 1600 1606}%
\special{pa 1600 406}%
\special{dt 0.045}%
\special{pa 1800 406}%
\special{pa 1800 1806}%
\special{dt 0.045}%
\special{pa 1800 2006}%
\special{pa 1800 3206}%
\special{dt 0.045}%
\special{pa 2400 3206}%
\special{pa 2400 2606}%
\special{dt 0.045}%
\special{pa 2400 1806}%
\special{pa 2400 406}%
\special{dt 0.045}%
\special{pa 2800 406}%
\special{pa 2800 1606}%
\special{dt 0.045}%
\special{pa 2800 2206}%
\special{pa 2800 3206}%
\special{dt 0.045}%
\special{pa 3200 3206}%
\special{pa 3200 406}%
\special{dt 0.045}%
\special{pa 3600 406}%
\special{pa 3600 1806}%
\special{dt 0.045}%
\special{pa 3600 1806}%
\special{pa 3600 2006}%
\special{dt 0.045}%
\special{pa 3600 2406}%
\special{pa 3600 3206}%
\special{dt 0.045}%
}}%
\put(32.2000,-33.6500){\makebox(0,0){\mbox{\Large $x_1$}}}%
\put(6.3000,-15.9500){\makebox(0,0){\mbox{\Large $x_2$}}}%
\put(40.0000,-33.6500){\makebox(0,0){\mbox{\Large $n$}}}%
\put(8.0000,-33.6500){\makebox(0,0){\mbox{\Large $-n$}}}%
\put(5.9000,-32.0500){\makebox(0,0){\mbox{\Large $-n$}}}%
\put(5.9000,-4.0500){\makebox(0,0){\mbox{\Large $n$}}}%
%
{\color[named]{Black}{%
\special{pn 13}%
\special{pa 3560 2960}%
\special{pa 3640 3040}%
\special{fp}%
\special{pa 3640 2960}%
\special{pa 3560 3040}%
\special{fp}%
}}%
%
{\color[named]{Black}{%
\special{pn 13}%
\special{pa 3560 960}%
\special{pa 3640 1040}%
\special{fp}%
\special{pa 3640 960}%
\special{pa 3560 1040}%
\special{fp}%
}}%
%
{\color[named]{Black}{%
\special{pn 13}%
\special{pa 3160 560}%
\special{pa 3240 640}%
\special{fp}%
\special{pa 3240 560}%
\special{pa 3160 640}%
\special{fp}%
}}%
%
{\color[named]{Black}{%
\special{pn 13}%
\special{pa 1760 2760}%
\special{pa 1840 2840}%
\special{fp}%
\special{pa 1840 2760}%
\special{pa 1760 2840}%
\special{fp}%
}}%
%
{\color[named]{Black}{%
\special{pn 13}%
\special{pa 1760 2960}%
\special{pa 1840 3040}%
\special{fp}%
\special{pa 1840 2960}%
\special{pa 1760 3040}%
\special{fp}%
}}%
%
{\color[named]{Black}{%
\special{pn 13}%
\special{pa 1160 560}%
\special{pa 1240 640}%
\special{fp}%
\special{pa 1240 560}%
\special{pa 1160 640}%
\special{fp}%
}}%
\end{picture}%